\documentclass[12pt]{article}
\usepackage[latin1]{inputenc}

\usepackage{amscd}
\usepackage{amsfonts}
\usepackage{amsgen}
\usepackage{amsmath}
\usepackage{mathabx}
\usepackage{amssymb}
\usepackage{amstext}
\usepackage{mathtools}
\usepackage{amsthm}
\usepackage{graphicx}
\usepackage{tikz-cd}
\usepackage{indentfirst}
\usepackage{latexsym}
\usepackage{makeidx}
\usepackage{mathrsfs}
\usepackage{epsfig}
\usepackage{wrapfig}
\usepackage[all,knot,arc,import,poly]{xy}
\usepackage{times}
\usepackage{amsmath}
\usepackage{amscd}
\usepackage{calrsfs}
\usepackage{color}
\usepackage{amsfonts}
\usepackage{amssymb}
\usepackage{mathabx}
\usepackage{url}
\usepackage{hyperref}
\usepackage[all]{xy}

\newtheorem{theorem}{Theorem}[section]
\newtheorem{proposition}[theorem]{Proposition}
\newtheorem{corollary}[theorem]{Corollary}

\newtheorem{fact}[theorem]{Fact}
\newtheorem{example}[theorem]{Example}
\newtheorem{definition}[theorem]{Definition}
\newtheorem{remark}[theorem]{Remark}

\newcommand{\R}{\mathbb{R}}

\newcommand{\ci}{\mathcal{C}^{\infty}}
\date{}
\makeatletter
\newcommand\footnoteref[1]{\protected@xdef\@thefnmark{\ref{#1}}\@footnotemark}
\makeatother

\title{\sc On the order theory for $\ci-$reduced $\mathcal{C}^{\infty}-$Rings and applications}

\author{}

\author{Jean Cerqueira Berni, IME-USP, \small{jeancb@ime.usp.br}, \\
Rodrigo Figueiredo, IME-USP  \small{rodrigof@ime.usp.br},\\
Hugo Luiz Mariano, IME-USP, {\small hugomar@ime.usp.br}}
\date{}

\begin{document}
\maketitle

\begin{center}

\begin{abstract}
\pagestyle{empty}
In the present work
 we carry on the study of the order theory for ($\ci$-reduced) $\ci$-rings initiated in \cite{rings1} (see also \cite{BM2}). In particular, we apply some results of the order theory of $\ci$-fields ({\it e.g.} every such field is real closed) to present another approach to the order theory of general $\ci$-rings: ``smooth real spectra''  (see \cite{separation}). This suggests that a model-theoretic investigation of the class of $\ci$-fields could be interesting and also useful to provide the first steps towards  the development of the ``Real Algebraic Geometry'' of $\ci$-rings.  
\end{abstract}
\end{center}


\section*{Introduction}


Given a smooth manifold, $M$, the set $\ci(M,\R)$  supports a far richer structure than just of an $\R$-algebra: it interprets not only the symbols for real polynomial functions  but for all smooth real functions $\R^n \to \R$, $n \in \mathbb{N}$. Thus,  $\ci(M,\R)$ is a natural instance of the algebraic structure called \emph{$\ci$-ring}.

It was not until the decades of 1970's and 1980's that a  study of the  abstract (algebraic) theory  of $\ci$-rings was made,  mainly  in order  to construct  topos models for ``Synthetic Differential Geometry'' (\cite{Kock}). The interest in $\ci$-rings gained strength in recent years, mainly motivated by the  differential version of `Derived Algebraic Geometry'' (see \cite{Joyce}). 

In this paper we address the study of the order theory of $\ci-$reduced $\ci-$rings, presenting a useful characterization of the ``natural order'' of a $\ci-$ring, introduced by Moerdijk and Reyes in \cite{rings1}: given any $\ci$-ring $\mathfrak{A} = (A, \Phi)$, this canonical strict partial order $\prec$ is given by:
$$(a \prec b) \iff (\exists u \in A^{\times})(b-a = u^2).$$

Since this natural binary relation given on a generic $\ci$-ring involves invertible elements, we should first analyze these elements of a $\ci$-ring. In order to do so, we shall restrict ourselves to the case of the {\em  $\ci$-reduced} $\ci$-rings. This is carried out in two steps: first proving the results for finitely generated $\ci$-rings and then proving them for arbitrary ones. 

Since any $\ci$-ring can be expressed as the quotient of a free object -- $\ci(\R^E)$, for some set $E$ -- by some (ring-theoretic) ideal, it is appropriate to characterize the equality between their elements by making use of these ring-theoretic ideals. We show that in this context the canonical strict partial order of a generic $\ci$-ring, say $\ci(\R^E)/I$, can be characterized by properties concerning filters of zerosets of functions in $I$.

In \cite{rings1}, Moerdijk and Reyes prove that every $\ci$-field is real closed (cf. \textbf{Theorem 2.10}). 
This suggests that the class of $\ci$-fields is ``well behaved'' with respect to its model theory. 

We apply, in particular,  some results on the order theory of $\ci$-fields -- {\em e.g.}, every such field is real closed (cf. \textbf{Theorem 2.10} in \cite{rings1}). -- to present another approach to the order theory of general $\ci$-rings, introducing the so-called ``smooth real spectra''  (see \cite{separation}). This suggests that a model-theoretic study of the class of $\ci$-fields could be interesting and also useful to provide the first steps towards  the development of the ``Real Algebraic Geometry'' of $\ci$-rings in the vein of \cite{Robson}. \\

{\bf Overview of the paper:} 
In the first section we present some preliminary notions and results that are used (implicitly or explicitly) throughout the paper, such as the basic concept of $\ci-$ring and some features of the category of $\ci-$rings, $\ci-$fields, $\ci-$rings of fractions, $\ci-$radical ideals, $\ci-$reduced $\ci$-rings and some facts about the smooth Zariski spectrum. 
{\bf Section 2} is devoted to present some results that (dually) connects subsets of $\mathbb{R}^E$ to quotients of $\ci(\mathbb{R}^E)$: we present the characterizations of equalities and inequalities between elements of $\ci-$reduced $\ci-$rings, {\em i.e.}, $\ci-$rings of the form $A=\frac{\ci(\R^E)}{I}$ with $\sqrt[\infty]{I}=I$, by means of the filter of zerosets of functions of $I$, and we use a Galois connection between filters of zerosets of $\R^E$ and ideals of $\ci(\R^E)$ to show that there are bijections between the set of maximal filters on $\R^E$ and the set of maximal ideals of $\ci(\R^E)$.  In \textbf{Section 3} we develop a detailed study of the natural strict partial ordering $\prec$ (introduced first by Moerdijk and Reyes) defined on a non-trivial $\ci$-reduced $\ci$-ring, with the aid of the results established in the previous sections. 
 {\bf Section 4}  presents some interesting results on $\ci-$fields based on the results from section 3: for instance, every $\ci$-field has $\prec$ as its unique (strict) total ordering compatible with the operations $+$ and $\cdot$, thus being a Euclidean field (in fact, it is real closed); this is useful to analyse the concept  of ``real $\ci-$spectrum of a $\ci-$ring'', which seems to be the suitable notion to deal with a smooth version of Real Algebraic Geometry.
 Finally, \textbf{Section 5} brings some concluding remarks, pointing some possible applications of the order structure of a $\ci-$reduced $\ci-$ring to its model theory.






\section{Preliminaries}



In this section we present the ingredients of the theory of $\ci$-rings needed in the sequel of this work for the reader's convenience: we present the class of $\ci$-rings as the class of models of an algebraic theory, and we describe the main notions of ``Smooth Commutative Algebra of $\ci$-rings": smooth rings of fractions, $\ci$-radicals, $\ci$-saturation and the smooth Zariski spectra. The main references used here are \cite{rings1}, \cite{rings2}, \cite{BM1}, \cite{BM2}. 

\subsection{On the algebraic theory of $\ci$-Rings}

\hspace{0.5cm}In order to formulate and study the concept of $\mathcal{C}^{\infty}-$ring, we use a first order language $\mathcal{L}$ with a denumerable set of variables (${\rm \bf Var}(\mathcal{L}) = \{ x_1, x_2, \cdots, x_n, \cdots\}$), whose nonlogical symbols are the symbols of all $\mathcal{C}^{\infty}-$functions from $\mathbb{R}^m$ to $\mathbb{R}^n$, with $m,n \in \mathbb{N}$, \textit{i.e.}, the non-logical symbols consist only of function symbols, described as follows.

For each $n \in \mathbb{N}$, we have the $n-$ary \textbf{function symbols} of the set $\mathcal{C}^{\infty}(\mathbb{R}^n, \mathbb{R})$, \textit{i.e.}, $\mathcal{F}_{(n)} = \{ f^{(n)} | f \in \mathcal{C}^{\infty}(\mathbb{R}^n, \mathbb{R})\}$. Thus, the set of function symbols of our language is given by:
      $$\mathcal{F} = \bigcup_{n \in \mathbb{N}} \mathcal{F}_{(n)} = \bigcup_{n \in \mathbb{N}} \mathcal{C}^{\infty}(\mathbb{R}^n)$$
      Note that our set of constants is $\mathbb{R}$, since it can be identified with the set of all $0-$ary function symbols, \textit{i.e.}, ${\rm \bf Const}(\mathcal{L}) = \mathcal{F}_{(0)} = \mathcal{C}^{\infty}(\mathbb{R}^0) \cong \mathcal{C}^{\infty}(\{ *\}) \cong \mathbb{R}$.\\

The terms of this language are defined in the usual way as the smallest set which comprises the individual variables, constant symbols and $n-$ary function symbols followed by $n$ terms ($n \in \mathbb{N}$).

Although a $\ci$-ring may be defined in any finitely complete category, we shall restrict ourselves to the study of  $\mathcal{C}^{\infty}-$rings in ${\rm \bf Set}$.

Before we proceed, we give the following:

\begin{definition}\label{cabala} A \textbf{$\mathcal{C}^{\infty}-$structure} on a set $A$ is a pair $ \mathfrak{A} =(A,\Phi)$, where:

$$\begin{array}{cccc}
\Phi: & \bigcup_{n \in \mathbb{N}} \mathcal{C}^{\infty}(\mathbb{R}^n, \mathbb{R})& \rightarrow & \bigcup_{n \in \mathbb{N}} {\rm Func}\,(A^n; A)\\
      & (f: \mathbb{R}^n \stackrel{\mathcal{C}^{\infty}}{\to} \mathbb{R}) & \mapsto & \Phi(f) := (f^{A}: A^n \to A)
\end{array},$$

\noindent that is, $\Phi$ interprets the \textbf{symbols} of all smooth real functions of $n$ variables as $n-$ary functions on $A$. Given two $\ci$-structures, $ \mathfrak{A} = (A, \Phi)$ and $\mathfrak{B} = (B, \Psi)$, a \textbf{$\ci-$structure homomorphism} is a function $\varphi: A \to B$ such that for any $n \in \mathbb{N}$ and any $f: \mathbb{R}^n \stackrel{\mathcal{C}^{\infty}}{\to} \mathbb{R}$ the following diagram commutes:
$$\xymatrixcolsep{5pc}\xymatrix{
A^n \ar[d]_{\Phi(f)}\ar[r]^{\varphi^{(n)}} & B^n \ar[d]^{\Psi(f)}\\
A \ar[r]^{\varphi^{}} & B
}$$
 \textit{i.e.}, $\Psi(f) \circ \varphi^{(n)} = \varphi^{} \circ \Phi(f)$. The class of $\ci-$structures and their morphisms compose a category that we denote by $\ci{\rm \bf Str}$.
\end{definition}

We call a $\mathcal{C}^{\infty}-$structure $\mathfrak{A} = (A, \Phi)$ a \textbf{$\mathcal{C}^{\infty}-$ring} if it preserves  projections and all equations between smooth functions. Formally, we have the following:

\begin{definition}\label{CravoeCanela}Let $\mathfrak{A}=(A,\Phi)$ be a $\mathcal{C}^{\infty}-$structure. We say that $\mathfrak{A}$ (or, when there is no danger of confusion, $A$) is a \textbf{$\mathcal{C}^{\infty}-$ring} if the following is true:\\

$\bullet$ Given any $n,k \in \mathbb{N}$ and any projection $p_k: \mathbb{R}^n \to \mathbb{R}$, we have:

$$\mathfrak{A} \models (\forall x_1)\cdots (\forall x_n)(p_k(x_1, \cdots, x_n)=x_k)$$

$\bullet$ For every $f, g_1, \cdots g_n \in \mathcal{C}^{\infty}(\mathbb{R}^m, \mathbb{R})$ with $m,n \in \mathbb{N}$, and every $h \in \mathcal{C}^{\infty}(\mathbb{R}^n, \mathbb{R})$ such that $f = h \circ (g_1, \cdots, g_n)$, one has:
$$\mathfrak{A} \models (\forall x_1)\cdots (\forall x_m)(f(x_1, \cdots, x_m)=h(g(x_1, \cdots, x_m), \cdots, g_n(x_1, \cdots, x_m)))$$

Given two $\ci-$rings, $\mathfrak{A}=(A, \Phi)$ and $\mathfrak{B}=(B, \Psi)$, a \textbf{$\ci$-homomorphism} is just a $\ci$-structure homomorphism between these $\ci-$rings. The category of all $\ci-$rings and $\ci-$ring homomorphisms make up a full subcategory of $\ci{\rm \bf Str}$, that we denote by $\ci{\rm \bf Rng}$.
\end{definition}


\begin{remark}[cf. \textbf{Sections 2, 3} and {\bf 4} of \cite{BM1}]
Since $\mathcal{C}^{\infty}{\rm \bf Rng}$ is a ``variety of algebras'' (it is a class of $\mathcal{C}^{\infty}-$structures which satisfy a given set of equations), it is closed under substructures, homomorphic images and products, by \textbf{Birkhoff's HSP Theorem}. Moreover:

$\bullet$ $\mathcal{C}^{\infty}{\rm \bf Rng}$ is a concrete category and the forgetful functor, $ U :  \mathcal{C}^{\infty}{\rm \bf Rng}  \to {\rm \bf Set}$ creates directed inductive colimits. Since $\mathcal{C}^{\infty}{\rm \bf Rng}$ is a variety of algebras, it has all (small) limits and (small) colimits. In particular, it has binary coproducts, that is, given any two $\mathcal{C}^{\infty}-$rings $A$ and $B$, we have their coproduct $A \stackrel{\iota_A}{\rightarrow} A\otimes_{\infty} B \stackrel{\iota_B}{\leftarrow} B$ again in $\ci{\rm \bf Rng}$;

$\bullet$ Each set $X$ freely generates a $C^\infty$-ring, $L(X)$, as follows:

-  for any finite set $X'$ with $\sharp X' = n$ we have $ L(X')= \mathcal{C}^{\infty}(\mathbb{R}^{X'}) \cong \mathcal{C}^\infty(\mathbb{R}^n, \mathbb{R})$, which is the free $C^\infty$-ring on $n$ generators, $n \in \mathbb{N}$;\\
-  for a general set, $X$, we take $L(X) = \mathcal{C}^{\infty}(\mathbb{R}^X):= \varinjlim_{X' \subseteq_{\rm fin} X} \mathcal{C}^{\infty}(\mathbb{R}^{X'})$;

$\bullet$ Given any $\mathcal{C}^{\infty}-$ring $A$ and a set, $X$, we can freely adjoin the set $X$ of variables to $A$ with the following construction: $A\{ X\}:= A \otimes_{\infty} L(X)$. The elements of $A\{ X\}$ are usually called $\mathcal{C}^{\infty}-$polynomials;

$\bullet$ The congruences of $\mathcal{C}^{\infty}-$rings are classified by their ``ring-theoretical'' ideals;

$\bullet$ Every $\mathcal{C}^{\infty}-$ring is the homomorphic image of some free $\mathcal{C}^{\infty}-$ring determined by some set, being isomorphic to the quotient of a free $\mathcal{C}^{\infty}-$ring by some ideal.

\end{remark}




Within the category of $\mathcal{C}^{\infty}-$rings we can perform a construction that is similar to the ``ring of fractions'' in Commutative Algebra, as well as define a suitable notion of ``radical ideal''. We  analyze these concepts in the following section.

\subsection{On $\ci$-Rings of Fractions and $\ci$-Radical Ideals}

In order to extend the notion of the ring of fractions to the category $\mathcal{C}^{\infty}{\rm \bf Rng}$, we make use of the universal property a ring of fractions must satisfy in ${\rm \bf Ring}$- except that we must deal with $\ci-$rings and $\ci-$homo\-morphisms instead of rings and homomorphisms of rings.

\begin{definition}\label{Alem}Let $\mathfrak{A} = (A,\Phi)$ be a $\mathcal{C}^{\infty}-$ring and $S \subseteq A$ one of its subsets. The $\mathcal{C}^{\infty}-$\textbf{ring of fractions} of $A$ with respect to $S$ is a $\mathcal{C}^{\infty}-$ring $A\{ S^{-1}\}$, together with a $\mathcal{C}^{\infty}-$homo\-morphism $\eta_S: A \to A\{ S^{-1}\}$ satisfying the following properties:
\begin{itemize}
  \item[(1)]{$(\forall s \in S)(\eta_S(s) \in (A\{ S^{-1}\})^{\times})$}
  \item[(2)]{If $\varphi: A \to B$ is any $\mathcal{C}^{\infty}-$homomorphism such that for every $s \in S$ we have $\varphi(s) \in B^{\times}$, then there is a unique $\mathcal{C}^{\infty}-$homomorphism $\widetilde{\varphi}: A\{ S^{-1}\} \to B$ such that the following triangle commutes:
      $$\xymatrixcolsep{5pc}\xymatrix{
      A \ar[r]^{\eta_S} \ar[rd]^{\varphi} & A\{ S^{-1}\} \ar[d]^{\widetilde{\varphi}}\\
        & B}$$}
\end{itemize}

By this universal property, the $\mathcal{C}^{\infty}-$ring of fractions is unique, up to (unique) isomorphisms.
\end{definition}

The existence of smooth rings of fractions can be guaranteed by a combination of constructions:

$\bullet$ first consider the addition of $\sharp  S$-variables to the $\ci$-ring $A$:

$$A\{ x_s | s \in S \} := A \otimes_{\infty} \mathcal{C}^{\infty}(\mathbb{R}^S),$$

and let $j_S : A \to A\{ x_s | s \in S \}$ be the (left) canonical morphism;

$\bullet$  now consider the ideal
$\langle \{ x_s \cdot \iota_A(s) - 1 | s \in S\}\rangle$  of $A$ generated by $\{ x_s \cdot \iota_A(s) - 1 | s \in S\}$, and take the quotient:

$$ A\{ x_s | s \in S \} \overset{q_S}\twoheadrightarrow \dfrac{A\{x_s | s \in S \}}{\langle \{ x_s \cdot \iota_A(s) - 1 | s \in S \}\rangle}.$$

Finally, define:

$$A\{S^{-1}\} := \frac{A\{x_s | s \in S \}}{\langle \{ x_s \cdot \iota_A(s) - 1 | s \in S \}\rangle}; $$ 
and 

$$\eta_S := q_s \circ j_s : A \to A\{S^{-1}\}.$$

It is not difficult to see that such a construction satisfies the required universal property.

\begin{example} \label{frac-exa}

Let $\varphi \in \ci(\R^n)$ and consider the (closed) subset  $Z(\varphi) = \{ \vec{x} \in \R^n: \varphi(\vec{x}) = 0 \} \subseteq \R^n$. Then $\ci(\R^n)\{\varphi^{-1}\} \cong \ci(\R^{n+1})/\langle \{ y \cdot \varphi -1 \}\rangle \cong \ci(\R^n \setminus Z(\varphi))$ and the restriction  map $\ci(\R^n) \to \ci(\R^n \setminus Z(\varphi))$ is a $\ci$-homomorphism that satisfies the universal property of $\eta_{\{\varphi\}}$.

\end{example}

Now we analyze the concept of the ``$\mathcal{C}^{\infty}-$radical ideal'' in the theory of $\mathcal{C}^{\infty}-$rings, which plays a similar role to the one played by radical ideals in Commutative Algebra. This concept was first presented by I. Moerdijk and G. Reyes in \cite{rings1} in 1986, and explored in more details in \cite{rings2}. \\

Unlike many notions in the branch of Smooth Rings such as $\mathcal{C}^{\infty}-$fields ($\ci$-rings whose underlying rings are fields), $\mathcal{C}^{\infty}-$do\-mains ($\ci$-rings whose underlying rings are domains) and local $\mathcal{C}^{\infty}-$rings ($\ci$-rings whose underlying rings are local rings), the concept of a $\mathcal{C}^{\infty}-$radical of an ideal cannot be brought from Commutative Algebra via the forgetful functor. This happens because when we take the localization of a $\ci$-ring by an arbitrary prime ideal, it is not always true that we get a local $\ci-$ring (see \textbf{Example 1.2} of \cite{rings2}). In order to get a local $\ci-$ring we must require an extra condition, that we are going to see later on.

Recall, from Commutative Algebra, that the radical of an ideal $I$ of a commutative unital ring $R$ is given by:

$$\sqrt{I} = \{ x \in R | (\exists n \in \mathbb{N})(x^n \in I)\}.$$

There are several characterizations of this concept, among which we highlight the following ones:

$$\sqrt{I} = \bigcap \{ \mathfrak{p} \in {\rm Spec}\,(R) | I \subseteq \mathfrak{p}  \} = \{ x \in R | \left( \dfrac{R}{I}\right)[(x+I)^{-1}] \cong 0\}.$$

The latter equality is the one which motivates our next definition.

\begin{definition}\label{defrad} ({\rm cf. p. 329 of \cite{rings1}}) Let $A$ be a $\mathcal{C}^{\infty}-$ring and let $I \subseteq A$ be a proper ideal. The \index{$\mathcal{C}^{\infty}-$radical}\textbf{$\mathcal{C}^{\infty}-$radical of $I$} is given by:

$$\sqrt[\infty]{I}:= \{ a \in A | \left( \dfrac{A}{I}\right)\{ (a+I)^{-1}\} \cong 0\}$$
\end{definition}

\begin{definition}\label{sat-def} ({\rm cf. {\bf Definition 2.1.5} of \cite{tese}})
Given a $\ci$-ring $A$ and a subset $S \subseteq A$, we define the {\bf $\ci-$saturation} of $S$ by:

$$S^{\infty-{\rm sat}} := \{ a \in A \mid \eta_S(a) \in A^{\times}\}.$$

\end{definition}

\begin{example} \label{sat-exa}

Given $\varphi \in \ci(\R^n)$, we have $\{\varphi\}^{\infty-{\rm sat}} = \{ \psi \in \ci(\R^n) \mid Z(\psi) \subseteq Z(\varphi)\}$. 

\end{example}

The concept of $\ci-$saturation is similar to the ordinary (ring-theoretic) concept of saturation in many aspects (for a detailed account of this concept, see \cite{BM2}). In particular, we use it to give a characterization of the elements of $\ci-$radical ideals.

\begin{proposition}\label{alba}{\rm [ \bf Proposition 19} of {\rm \cite{BM2}]} Let $A$ be a $\mathcal{C}^{\infty}-$ring and let $I \subseteq A$ be any ideal. We have the following equalities:
  $$\sqrt[\infty]{I} = \{ a \in A | (\exists b \in I)\& (\eta_a(b) \in (A\{ a^{-1}\})^{\times}) \} = \{ a \in A | I \cap \{ a\}^{\infty-{\rm sat}} \neq \varnothing\}$$
  where $\eta_a : A \to A\{ a^{-1}\}$ is the $\ci-$homomorphism of fractions with respect to $\{ a\}$.
\end{proposition}

In ordinary Commutative Algebra, given an element $x$ of a ring $R$, we say that $x$ is a nilpotent infinitesimal if and only if there is some $n \in \mathbb{N}$ such that $x^n=0$. Let $A$ be a  $\mathcal{C}^{\infty}-$ring and $a \in A$. D. Borisov and K. Kremnizer in \cite{Borisov} call $a$ an $\infty-$infinitesimal if, and only if $A\{ a^{-1}\} \cong 0$. The next definition describes the notion of a $\mathcal{C}^{\infty}-$ring being free of $\infty-$infinitesimals - which is analogous to the notion of ``reducedness'', of a commutative ring.\\

\begin{definition}A $\mathcal{C}^{\infty}-$ring $A$ is \index{$\mathcal{C}^{\infty}-$reduced}\textbf{$\mathcal{C}^{\infty}-$reduced} if, and only if, $\sqrt[\infty]{(0)} = (0)$.
\end{definition}

\begin{example}\label{free} The simplest example of $\ci-$reduced $\ci-$rings is the free $\ci-$rings on any set of generators $E$ (cf. {\rm \bf Proposition 37} of \cite{BM2}).
\end{example}

Next we  register some useful results  on $\ci-$radical ideals and $\ci-$reduced $\ci-$rings.

\begin{proposition}[{\rm {\bf Proposition 31}, \cite{BM2}}] \label{redi}
Let $A', B'$ be two $\mathcal{C}^{\infty}-$rings and $\jmath: A' \to B'$ be a monomorphism. If $B'$ is $\mathcal{C}^{\infty}-$reduced, then $A'$ is also $\mathcal{C}^{\infty}-$reduced.
\end{proposition}


\begin{proposition}  \label{rad-pr}
Let $A$ be a $\mathcal{C}^{\infty}-$ring. We have:
\begin{itemize}
  \item[(a)]{An ideal $J \subseteq A$  is a $\mathcal{C}^{\infty}-$radical ideal if, and only if, $\dfrac{A}{J}$ is a $\mathcal{C}^{\infty}-$reduced $\mathcal{C}^{\infty}-$ring}.
  \item[(b)]{A proper prime ideal $\mathfrak{p} \subseteq A$ is $\mathcal{C}^{\infty}-$radical if, and only if, $\dfrac{A}{\mathfrak{p}}$ is a $\mathcal{C}^{\infty}-$reduced $\mathcal{C}^{\infty}-$domain.}
\end{itemize}
\end{proposition}
\begin{proof}
See \textbf{Corollary 10} of \cite{BM2}. 
\end{proof}

Next we present some properties of $\mathcal{C}^{\infty}-$radical ideals of a $\mathcal{C}^{\infty}-$ring $A$ regarding some ``operations'' such as the intersection,  the directed union and the preimage by a $\ci-$homomorphism of $\ci-$radical ideals. To simplify the notation, given a $\ci-$ring $A$, we denote by $\mathfrak{I}^{\infty}_{A}$ the set of all its $\ci-$radical ideals. The proofs of the results given in the next proposition can be found in \cite{BM2}.

\begin{proposition}\label{egito}The following results hold:
\begin{itemize}
\item[(a)]{Suppose that $(\forall \alpha \in \Lambda)(I_{\alpha} \in \mathfrak{I}^{\infty}_A)$. Then $\bigcap_{\alpha \in \Lambda} I_{\alpha} \in \mathfrak{I}^{\infty}_{A}$, that is, if $(\forall \alpha \in \Lambda)(I_{\alpha} \in \mathfrak{I}^{\infty}_A)$, then:
    $$\sqrt[\infty]{\bigcap_{\alpha \in \Lambda}I_{\alpha}} = \bigcap_{\alpha \in \Lambda}I_{\alpha} = \bigcap_{\alpha \in \Lambda} \sqrt[\infty]{I_{\alpha}}$$}
\item[(b)]{Let $\{ I_{\alpha} | \alpha \in \Sigma \}$ an upwards directed family of elements of $\mathfrak{I}^{\infty}_{A}$. Then $\bigcup_{\alpha \in \Sigma} I_{\alpha} \in \mathfrak{I}^{\infty}_{A}$.}
\item[(c)]{Let $A,B$ be $\mathcal{C}^{\infty}-$rings, $f: A \to B$ a $\mathcal{C}^{\infty}-$homomorphism and $J \subseteq B$ any ideal. Then:
$$\sqrt[\infty]{f^{\dashv}[J]} \subseteq f^{\dashv}[\sqrt[\infty]{J}].$$}
\item[(d)]{Let $A,B$ be $\mathcal{C}^{\infty}-$rings, $f: A \to B$ be a $\mathcal{C}^{\infty}-$homomorphism and $J \subseteq B$ be a $\ci-$radical ideal. Then $f^{\dashv}[J]$ is a $\ci$-radical ideal of $A$.}
\item[(e)]{Given any two $\mathcal{C}^{\infty}-$radical ideals of a $\mathcal{C}^{\infty}-$ring $A$, $I,J \in \mathfrak{I}^{\infty}_{A}$, we have:
$$\sqrt[\infty]{I \cdot J} = \sqrt[\infty]{I \cap J}$$}
\end{itemize}
\end{proposition}

For a more comprehensive account of $\ci-$radical ideals and $\ci-$reduced $\ci$-rings, we refer the reader to \cite{BM2}.

\subsection{On the Smooth Zariski Spectrum}

Recall that the spectrum of a commutative unital ring $R$ consists of all prime ideals of $R$, together with a spectral topology - given by its ``distinguished basic sets'', its Zariski topology. Recall, also, that in ordinary Commutative Algebra, every prime ideal is radical - and that the $\ci-$version of this implication is false in the context of Smooth Commutative Algebra (not every prime ideal of a $\ci-$ring is $\ci-$radical). At this point it is natural to look for a $\mathcal{C}^{\infty}-$analog of the Zariski spectrum of a commutative unital ring. Keeping in mind the definitions of the previous subsection, we give the following definition, that can be found in \cite{rings2}:

\begin{definition}For a $\mathcal{C}^{\infty}-$ring $A$, we define:
$${\rm Spec}^{\infty}\,(A) = \{ \mathfrak{p} \in {\rm Spec}\,(A) | \mathfrak{p} \, \mbox{is} \, \mathcal{C}^{\infty}-\mbox{radical} \}$$
equipped with the smooth Zariski topology defined by its basic open sets:
$$D^{\infty}(a) = \{ \mathfrak{p} \in {\rm Spec}^{\infty}\,(A) | a \notin \mathfrak{p} \} $$
It is proved (see {\rm \cite{BM2}}) that ${\rm Spec}^{\infty}\,(A)$, for any $\ci-$ring $A$, is a spectral space.
\end{definition}


\begin{proposition}[{\rm {\bf Proposition 48}}, \cite{BM2}]Let $A,A'$ be two $\mathcal{C}^{\infty}-$rings and let $f: A \to A'$ be a $\mathcal{C}^{\infty}-$homo\-morphism. The function:
$$\begin{array}{cccc}
    h^{*}: & {\rm Spec}^{\infty}\,(A') & \rightarrow & {\rm Spec}^{\infty}\,(A) \\
     & \mathfrak{p} & \mapsto & h^{\dashv}[\mathfrak{p}]
  \end{array}$$
is a spectral map.
\end{proposition}

\begin{theorem} (\textbf{Separation Theorems}, {\rm \cite{separation}})\label{TS} Let $A$ be a $\mathcal{C}^{\infty}-$ring, $S \subseteq A$ be a subset of $A$ and $I$ be an ideal of $A$. Denote by $\langle S \rangle$ the multiplicative submonoid of $A$ generated by $S$. We have:
\begin{itemize}
  \item[(a)]{If $I$ is a  $\mathcal{C}^{\infty}-$radical ideal, then:
$$I \cap \langle S \rangle = \varnothing \iff {I} \cap S^{\infty-{\rm sat}} = \varnothing$$}
  \item[(b)]{If  $S \subseteq A$ is a $\mathcal{C}^{\infty}$-saturated subset, then:
$$I \cap S = \varnothing \iff \sqrt[\infty]{I} \cap S = \varnothing$$
}
 \item[(c)]{If $\mathfrak{p} \in {\rm Spec}^{\infty}\,(A)$, then $A\setminus \mathfrak{p} = (A \setminus \mathfrak{p})^{\infty-{\rm sat}}$}
 \item[(d)]{If  $S \subseteq A$ is a $\mathcal{C}^{\infty}$-saturated subset, then:
$$I \cap S = \varnothing \iff (\exists \mathfrak{p} \in {\rm Spec}^{\infty}\,(A))((I \subseteq \mathfrak{p})\& (\mathfrak{p} \cap S = \varnothing)).$$}
 \item[(e)]{$\sqrt[\infty]{I} = \bigcap \{ \mathfrak{p} \in {\rm Spec}^{\infty}\,(A) | I \subseteq \mathfrak{p} \}$}
\end{itemize}
\end{theorem}

A more detailed account of the smooth Zariski spectrum  containing detailed proofs can be found in \textbf{Section 5} of \cite{BM2}




\section{On smooth spaces and smooth algebras}

Every (finite dimensional) smooth manifold $M$ can be embedded as a closed subspace of some $\R^n$ (Whitney's Theorem) and  determines a $\ci$-ring, $\ci(M)$. This mapping, $M \mapsto  \ci(M)$,  extends to a full and faithful contravariant functor into the category of $\ci$-rings (see for instance {\bf Theorem 2.8} of \cite{MRmodels}). In this section we present some results that (dually)  connects subsets of $\mathbb{R}^E$ and quotients of $\ci(\mathbb{R}^E)$. More precisely, we present some very useful characterizations of equalities and inequalities between elements of $\ci-$reduced $\ci-$rings, {\em i.e.}, $\ci-$rings of the form $A=\frac{\ci(\R^E)}{I}$ with $\sqrt[\infty]{I}=I$, by means of the filter of zerosets of functions of $I$.

\subsection{The finitely generated case}

We start by recalling an important fact about the relation between closed subsets of $\R^n$ and zerosets of $\ci$-functions.

\begin{fact} \label{char} (essentially {\bf Lemma 1.4} of \cite{MRmodels}) For each open subset $U \subseteq \R^n$ there is a smooth function $\chi : \R^n \to \R$ such that:
\begin{itemize}
    \item $(\forall x \in \R^n)(\chi(x) \geq 0)$
    
    \item $(\forall x \in \R^n)((\chi(x) = 0)  \iff (x \notin U))$.
\end{itemize}

\end{fact}

Recall that a function $f : X \to \R$ defined over an arbitrary subset $X$ of some $\R^n$ is called {\em smooth} if there is an open subset $U \subseteq \R^n$ such that $X \subseteq U$  and a $\ci$-extension of $f$,  $\widetilde{f} : U \to \R$, such that $\widetilde{f}\upharpoonright_X = f$.

\begin{fact}[{\rm \bf Smooth Tietze Theorem}]\label{obs}  Let $F \subseteq \mathbb{R}^n$ be a {\em closed} set and let $f \in \mathcal{C}^{\infty}(F)$. Then  there is a smooth function  $\widetilde{f} \in \ci(\R^n)$ such that  $\widetilde{f}\upharpoonright_F = f$. Moreover:

\begin{itemize}

 \item If  $(\forall x \in F)(f(x) \neq 0)$, then we can choose a $\ci$-extension $\widetilde{f}$ of $f$  and an open subset $U \subseteq \R^n$ such that  $F \subseteq U$ and  $(\forall x \in U)(\widetilde{f}(x) \neq  0)$.

\item If  $(\forall x \in F)(f(x) > 0)$, then we can  choose a $\ci$-extension $\widetilde{f}$ of $f$ such that   $(\forall x \in \R^n)(\widetilde{f}(x) >  0)$.

\end{itemize} 

\end{fact}

\begin{proposition}\label{equacional}Let  $A = \frac{\ci(\mathbb{R}^n)}{I}$ be a $\ci-$reduced $\ci-$ring, so $\sqrt[\infty]{I}=I$. Given $f,g \in \ci(\mathbb{R}^n)$, we have:

$$(q_I(f) = f+I = g+I = q_I(g)) \iff (\exists \varphi \in I)(\forall x \in Z(\varphi))(f(x)=g(x)).$$
\end{proposition}
\begin{proof}
Suppose $q_I(f) = f+I = g+I = q_I(g)$, so $g-f \in I$. It suffices to take $\varphi = g-f$, so:
$$(\forall x \in Z(\varphi))(g(x)-f(x)=0)$$
and
$$(\forall x \in Z(\varphi))(f(x)=g(x))$$
Conversely, suppose there is some $\varphi \in I$ such that $(\forall x \in Z(\varphi))(f(x)=g(x))$. Thus, $Z(\varphi) \subseteq Z(g-f)$ and $\varphi\upharpoonright_{\mathbb{R}^n \setminus Z(g-f)} \in \ci(\mathbb{R}^n \setminus Z(g-f))^{\times} \cong \ci(\mathbb{R}^n)\{ {(g-f)}^{-1}\}^{\times}$ (see {\bf Example \ref{frac-exa}}). 
Since there is $\varphi \in I$ such that $\varphi \in \ci(\mathbb{R}^n)\{ {(g-f)}^{-1}\}^{\times}$, it follows that $g-f \in \sqrt[\infty]{I} \subseteq I$ and $f+I = g+I$.
\end{proof}







Now we characterize the invertible elements of a $\ci$-reduced $\ci$-ring.

\begin{proposition}\label{invertivel}Let $A=\frac{\ci(\mathbb{R}^n)}{I}$ be a $\ci-$reduced finitely generated $\ci-$ring, so $\sqrt[\infty]{I}=I$. Given $f \in \ci(\R^n)$ we have:
$$\left(q_I(f)= (f+I) \in \left(\frac{\ci(\R^n)}{I}\right)^{\times}\right) \iff (\exists \varphi \in I)(\forall x \in Z(\varphi))(f(x)\neq 0).$$

\end{proposition}
\begin{proof}Suppose, first, that $(f+I) \in \left(\frac{\ci(\R^n)}{I}\right)^{\times}$, so there is some $h + I \in \frac{\ci(\R^n)}{I}$ such that:
$$(f+I)\cdot (h + I) = 1 + I$$
$$q_I(f)\cdot q_I(h) = q_I(1)$$
$$f\cdot h - 1 \in \ker (q_I)$$
$$\varphi = f \cdot h - 1 \in I,$$
One has:
$$(\forall x \in Z(\varphi))(f(x)\cdot h(x) =  1 \neq 0)$$
and thus:
$$(\forall x \in Z(\varphi))(f(x) = \frac{1}{h(x)}  \neq 0)$$

Conversely, suppose that $f \in \ci(\R^n)$ is such that there is some $\varphi \in I$ with $(\forall x \in Z(\varphi))(f(x)\neq 0)$. Since $f$ is a continuous function, there is an open subset $U \subseteq \R^n$ such that $Z(\varphi) \subseteq U$ and  $(\forall x \in U)({f}(x)\neq 0)$.


We define:
$$\begin{array}{cccc}
g:& U \subseteq \R^n & \rightarrow & \R\\
 & x & \mapsto & \frac{1}{f(x)}
\end{array}$$


Thus  $g\upharpoonright_{Z(\varphi)}$ is smooth on $Z(\varphi)$ and by \textbf{Smooth Tietze's Theorem} ({\bf Fact \ref{obs}}), one is able to construct a $\ci$-function $\widetilde{g}: \R^n \to \R$ such that $\widetilde{g}\upharpoonright_{Z(\varphi)} = g \upharpoonright_{Z(\varphi)}$.  

Since we have:
$$(\forall x \in Z(\varphi))(f(x)\cdot \widetilde{g}(x) - 1= 0)$$
it follows, by \textbf{Proposition \ref{equacional}}, that $f \cdot \widetilde{g}-1 \in I$, so:
$$(f+I)\cdot (\widetilde{g}+I) = 1+ I$$
and $f+I \in \left( \frac{\ci(\R^n)}{I} \right)^{\times}$.
\end{proof}

Combining the previous proposition with  {\bf Proposition \ref{equacional}}, we obtain the following:

\begin{corollary} \label{square-cor}
Let $A=\frac{\ci(\mathbb{R}^n)}{I}$ be a  finitely generated $\ci-$reduced $\ci-$ring, so $\sqrt[\infty]{I}=I$. Given $f \in \ci(\R^n)$,  are equivalent:

\begin{enumerate}
    \item $(\exists u \in \ci(\R^n))( ((f+I) = (u^2 +I)) \& (u+I \in \left(\frac{\ci(\R^n)}{I}\right)^{\times}))$
    \item $ (\exists u \in \ci(\R^n))( \exists \psi \in I) (\forall x \in Z(\psi))(f(x) = u^2(x) \neq 0)$
    \item  $(\exists \psi \in I)(\forall x \in Z(\psi))(f(x) > 0)$
\end{enumerate}

\end{corollary}




\subsection{The general case}


We know that every closed subset of $\R^n$ is the zeroset of some smooth function $f: \R^n \to \R$ (see {\bf Fact \ref{char}}). We now expand the notion of zeroset for $\R^E$, where $E$ is not  necessarily a finite set.

\begin{definition}
Let $E$ be \underline{any} set. Consider $\R^E = \{ v : E \to \R | v\, \text{is a function}\}$ and denote ${\cal F}(\R^E) : = \{ f : \R^E \to \R | f$ is a function$\}$.

For every  $D \subseteq E$, we 
have the canonical projection:
$$\begin{array}{cccc}
\pi_{ED}: & \R^E & \rightarrow & \R^D\\
      & v & \mapsto & v\upharpoonright_D: D \to \R
\end{array}$$
and this induces a function:
$$\begin{array}{cccc}
\mu_{DE}: &  {\cal F}(\R^D) & \rightarrow & {\cal{F}}(\R^E)\\
      & f & \mapsto & \mu_{DE}(f) := f \circ \pi_{ED} 
\end{array}$$





\end{definition}

\begin{definition}Let $E$ be any set. Define:

$$\ci(\R^E) := \{ f \in {\cal F}(\R^E) \mid (\exists E' \subseteq_{\rm fin} E)(\exists {f'} \in \ci(\R^{E'}))(f = {f'} \circ \pi_{EE'})\}.$$


\end{definition}

It is not hard to see that for every $E', E'' \subseteq_{\rm fin} E$ with $E' \subseteq E''$,  the following diagram commutes:

$$\xymatrix{
 & \ci(\R^{E}) & \\
\ci(\R^{E'}) \ar[ur]^{\mu_{E'E}} \ar[rr]_{\mu_{E'E''}}& & \ci(\R^{E''}) \ar[ul]_{\mu_{E''E}}
}$$

Moreover, notice that $\ci(\R^E) \cong \varinjlim_{E' \subseteq_f E} \ci(\R^{E'})$, where for every $E', E'' \subseteq_{\rm fin} E$ such that $E' \subseteq E''$, the following triangle commutes:

$$\xymatrix{
 & \varinjlim_{E' \subseteq_{fin} E} \ci(\R^{E'}) & \\
\ci(\R^{E'}) \ar[ur]^{\ell_{E'}} \ar[rr]_{{\mu}_{E'E''}}& & \ci(\R^{E''}) \ar[ul]_{\ell_{E''}}
},$$

\noindent where the morphisms indicated above are defined as in \textbf{Section 3} of \cite{BM1}.

Thus, by a smooth function on $\R^E$ we mean a function $f: \R^E \to \R$ that factors through some projection $\pi_{EE'}: \R^{E} \to \R^{E'}$ and a smooth function $\widetilde{f} \in \ci(\R^{E'})$, for some $E' \subseteq_{fin} E$. I.e., given $f \in {\cal{F}}(\R^E)$ we have:

$$f \in \ci(\R^E)  \iff (\exists E' \subseteq_{\rm fin} E)(\exists \widetilde{f} \in \ci(\R^{E'})(f = \widetilde{f} \circ \pi_{EE'}).$$

\begin{definition}
Let $E$ be \underline{any} set. A subset $X \subseteq \R^E$ is a \textbf{zeroset} if there is some $\varphi \in \ci(\R^E)$ such that $X = Z(\varphi)$, 
 where 
 $$Z(\varphi) :=  \varphi^{\dashv}[\{0\}] =  \{\vec{x} \in \R^{E} : \varphi(\vec{x}) = 0 \}$$

The set $\mathcal{Z}_E := \{Z(\varphi) \in \wp(\R^E):  \varphi \in \ci(\R^{E}) \}$ denotes  the set  of all zerosets in $\R^E$.

\end{definition}

\begin{remark}

$\bullet$ Let $E$ be an arbitrary set and  $ \varphi \in \ci(\R^E)$. Select $E' \subseteq_{fin} E$ and $\varphi' \in \ci(\R^{E'})$ such that $\varphi = \varphi'\circ \pi_{EE'}$. Then $Z(\varphi) = \pi_{EE'}^{\dashv}[Z(\varphi')]$.

$\bullet$ If $E$ is a {\em finite} set, then by {\bf Fact \ref{char}}, $\mathcal{Z}_E  = {\rm Closed}\,(\R^E) \subseteq \wp(\R^E)$ thus it is stable under {\em finite} reunions  and {\em arbitrary} intersections; in particular, $\emptyset = \bigcup \emptyset$ and $\R^E = \bigcap \emptyset$ are in $\mathcal{Z}_E$.

$\bullet$ In general, for an {\em arbitrary} set $E$, the subset $\mathcal{Z}_E  \subseteq \wp(\R^E)$  is stable just under {\em finite} reunions and {\em finite} intersections. 

\end{remark}

\begin{definition}

If $I \subseteq \ci(\R^E)$ is an ideal, then $I' = \mu_{E'E}^{\dashv}[I]$ is an ideal of $\ci(\R^{E'})$. We define:

$$\widehat{I} = \{ F \subseteq \R^E \mid (\exists E' \subseteq_{\rm fin} E)(\exists f \in I'= \mu_{E'E}^{\dashv}[I])(F = \pi_{EE'}^{\dashv}[Z(f)])\}$$
\end{definition}

\begin{proposition}\label{equacional2}Let  $A = \frac{\ci(\mathbb{R}^E)}{I}$ be a $\ci-$reduced $\ci-$ring, so $\sqrt[\infty]{I}=I$. Given $f,g \in \ci(\mathbb{R}^E)$, we have:

$$(q_I(f) = f+I = g+I = q_I(g)) \iff (\exists \varphi \in I)(\forall x \in Z(\varphi))(f(x)=g(x)).$$
\end{proposition}
\begin{proof}
Given $f,g \in \ci(\R^E)$ such that $q_I(f) = f+I = g+I = q_I(g)$, by definition there are finite subsets $E_f, E_g \subseteq E$ and $\widehat{f} \in \ci(\R^{E_f}), \widehat{g} \in \ci(\R^{E_g})$ such that $f= \mu_{E_f}(\widehat{f}) = \widehat{f} \circ \pi_{E_f} \in \ci(\R^{E_f})$ and $g= \mu_{E_g}(\widehat{g}) = \widehat{g} \circ \pi_{E_g} \in \ci(\R^{E_g})$. Then $E_f \cup E_g \subseteq_{fin} E$. Let $\widetilde{f} = \mu_{E_f, E_f \cup E_g}(\widehat{f}) \in \ci(\R^{E_f \cup E_g})$ and  $\widetilde{g} = \mu_{E_g, E_f \cup E_g}(\widehat{g})  \in \ci(\R^{E_f \cup E_g})$. By hypothesis, $f+I = g+I$, so $f-g \in I$ and $\mu_{E_f \cup E_g, E}(\widetilde{f})-\mu_{E_f \cup E_g, E}(\widetilde{g}) \in I$. We have, thus,  $(\widetilde{f}-\widetilde{g}) \in \mu_{E_f \cup E_g, E}^{\dashv}[I] = \sqrt[\infty]{\mu_{E_f \cup E_g}^{\dashv}[I]}$, since $I$ is a $\ci-$radical ideal (see {\bf Proposition \ref{egito}.(d)}). By the finitely generated case  ({\bf Proposition \ref{equacional}}), since $\widetilde{f}, \widetilde{g} \in \ci(\R^{E_f \cup E_g})$ and $\widetilde{f}+ \mu_{E_f \cup E_g}^{\dashv}[I] = \widetilde{g}+\mu_{E_f \cup E_g}^{\dashv}[I]$, it follows that there is some $\widetilde{\varphi} \in \mu_{E_f \cup E_g}^{\dashv}[I]$ such that:
$$(\forall y \in Z(\widetilde{\varphi}))(\widetilde{f}(y)=\widetilde{g}(y))$$
Taking $\varphi = \mu_{E_f \cup E_g,E}(\widetilde{\varphi}) = \widetilde{\varphi} \circ \pi_{E,E_f \cup E_g} \in I$, we have:
$$(\forall x \in Z(\varphi))(f(x)=\widetilde{f}\circ \pi_{E_f \cup E_g}(x) = \widetilde{g}\circ \pi_{E_f \cup E_g}(x) = g(x))$$

On the other hand, suppose $f,g \in \ci(\R^E)$ are such that $(\exists \varphi \in I)(\forall x \in Z(\varphi))(f(x)=g(x))$. Thus, for such $\varphi$ there is a finite $E_{\varphi} \subseteq E$ and $\widehat{\varphi} \in \ci(\R^{E_{\varphi}})$ such that $\varphi = \widehat{\varphi} \circ \pi_{E,E_{\varphi}}$, and there are also some finite $E_f, E_g \subseteq E$ and some $\widehat{f} \in \ci(\R^{E_f}), \widehat{g} \in \ci(\R^{E_g})$ such that $f = \mu_{E_f, E}(\widehat{f})$ and $g=\mu_{E_g,E}(\widehat{g})$. Let $\widetilde{\varphi} = \mu_{E_{\varphi}, E_{\varphi} \cup E_f \cup E_g}(\widehat{\varphi}), \widetilde{f} = \mu_{E_f, E_{\varphi} \cup E_f \cup E_g}(\widehat{f})$ and $\widetilde{g} = \mu_{E_g, E_{\varphi} \cup E_f \cup E_g}(\widehat{g})$. By the finitely generated case ({\bf Proposition \ref{equacional}}), since $\widetilde{f}, \widetilde{g}, \widetilde{\varphi} \in \ci(\R^{E_{\varphi} \cup E_f \cup E_g})$, $(\forall x \in Z(\widetilde{\varphi}))(\widetilde{f}(x)=\widetilde{g}(x))$ and $\sqrt[\infty]{\mu_{E_{\varphi} \cup E_f \cup E_g,E}^{\dashv}[I]} = \mu_{E_{\varphi} \cup E_f \cup E_g,E}^{\dashv}[I]$, it follows that $\widetilde{f}-\widetilde{g} \in \mu_{E_{\varphi} \cup E_f \cup E_g, E}^{\dashv}[I]$, so $f-g = \mu_{E_{\varphi} \cup E_f \cup E_g,E}(\widetilde{f}-\widetilde{g})  \in I$, and $f+I = g+I$.
\end{proof}

\begin{proposition}\label{invertivel2} Let $E$ be any set and $I \subseteq \ci(\R^E)$ be a $\ci$-radical ideal. We have, for every $ f \in \ci(\R^E)$:
$$(f+I \in \left( \dfrac{\ci(\R^E)}{I}\right)^{\times} )\iff (\exists \varphi \in I)(\forall x \in Z(\varphi))(f(x)\neq 0)$$
\end{proposition}
\begin{proof}
Given $f \in \ci(\R^E)$ such that $q_I(f) = f+I$ is invertible, let $h, \varphi \in \ci(\R^E)$ be such that 
$$(f \cdot h- 1) = \varphi \in I.$$
As in the proof of previous proposition, we can select $E' \subseteq_{fin} E$ and $f', h', \varphi' \in \ci(\R^{E'})$ such that $f = \mu_{E'E}(f'), h = \mu_{E'E}(h'), \varphi = \mu_{E'E}(\varphi')$. Then 
$$(f'\cdot h' -'1) = \varphi' \in I' := \mu_{E'E}^{\dashv}[I].$$ 
Thus 
$$(\forall x' \in \R^{E'})( x' \in Z(\varphi') \Rightarrow f'(x') \neq 0 )$$
Since $Z(\varphi) = \pi_{EE'}^{\dashv}[Z(\varphi')]$ and $f = f' \circ \pi_{EE'},$ then
$$(\forall x \in \R^{E})( x \in Z(\varphi) \Rightarrow f(x) \neq 0 )$$

Conversely,  let $f, \varphi \in \ci(\R^E)$ such that $\varphi \in I$ and  
$$(\forall x \in \R^{E})( x \in Z(\varphi) \to f(x) \neq 0 ).$$
Select $E' \subseteq_{fin} E$ and $f', \varphi' \in \ci(\R^{E'})$ such that $f = \mu_{E'E}(f'), \varphi = \mu_{E'E}(\varphi')$. 

Then $I' := \mu_{E'E}^{\dashv}[I]$ is a $\ci$-radical ideal of $\ci(\R^{E'})$, $\varphi' \in I'$ and  
$$(\forall x' \in \R^{E'})( x' \in Z(\varphi') \to f'(x') \neq 0 ).$$ 

By the finitely generated case  ({\bf Proposition \ref{invertivel}}),  $f'+I' \in (\ci(\R^{E'})/I')^\times$. Let $h' \in \ci(\R^{E'})$ such that 
$$(f'+I')(h'+I') = 1+I' \in \ci(\R^{E'})/I'.$$ 

Now define $h := \mu_{EE'}(h')$. Then
$$ (f+I)(h+I) = 1+I \in \ci(\R^{E})/I $$

\end{proof}

\begin{proposition}Let $E$ be any set. If $I \subseteq \ci(\R^E)$ is an ideal, then:
$$\widehat{I} := \{ X \in \wp(\R^E)   \mid (\exists f \in I)(X = Z(f))\} \subseteq \wp(Z(\R^{E}))$$
is a filter of zerosets in $\R^E$.
\end{proposition}
\begin{proof}
Note first that
$$\widehat{I} = \{ F \subseteq \R^E \mid (\exists E' \subseteq_{\rm fin} E)(\exists f' \in I'= \mu_{E'}^{\dashv}[I])(F = \pi_{EE'}^{\dashv}[Z(f)])\} \subseteq \wp({\cal Z}(\R^{E}))$$

It is easy to see that $\R^E$ is a zeroset. In fact, $\R^E = Z(0_E)$, where
$$\begin{array}{cccc}
0_E: & \R^E & \rightarrow & \R \\
   & x & \mapsto & 0
\end{array}.$$

Note that $0_E \in \ci(\R^E)$: choose any finite $D \subseteq E$ and consider the $\ci$-function

$$\begin{array}{cccc}
0_D: & \R^D & \rightarrow & \R \\
   & x & \mapsto & 0
\end{array},$$

\noindent so $0_E = \mu_{DE}(0_D) \in \ci(\R^E)$.

Given $G_1,G_2 \in \widehat{I} \subseteq \R^{E}$, let $g_1, g_2 \in I \subseteq \ci(\R ^E)$  such that $G_i = Z(g_i)$ for $i=1,2$.  There are  finite $E', E'' \subseteq E$, $f_1 \in \mu_{E'E}^{\dashv}[I] \subseteq \ci(\R^{E'}), f_2 \in \mu_{E''E}^{\dashv}[I] \subseteq \ci(\R^{E''})$  such that $g_1 = f_1 \circ \pi_{EE'}, g_2 = f_2 \circ \pi_{EE''}$. Thus $\pi_{EE'}^{\dashv}[Z(f_1)] = G_1$ and $\pi_{EE''}^{\dashv}[Z(f_2)] = G_2$, where $\pi_{EE'}: \R^{E} \to \R^{E'}$ and $\pi_{EE''}: \R^{E} \to \R^{E''}$ are the canonical projections (restrictions). Consider:

$$\xymatrix{
 & \ar[dl]_{\pi_{E'\cup E'', E'}} \R^{E' \cup E''}  \ar[dr]^{\pi_{E'\cup E'', E''}} &  \\
 \R^{E'} & & \R^{E''}}$$
 
 \noindent where:

 $$\begin{array}{cccc}
    \pi_{E'\cup E'', E'}: & \R^{E' \cup E''} & \rightarrow & \R^{E'}\\
          & v & \mapsto & v\upharpoonright_{E'}: E' \to \R
 \end{array}$$
 and
 
 $$\begin{array}{cccc}
    \pi_{E'\cup E'', E''}: & \R^{E' \cup E''} & \rightarrow & \R^{E''}\\
          & v & \mapsto & v\upharpoonright_{E''}: E'' \to \R
 \end{array}$$

We have the commutative diagram:

$$\xymatrixcolsep{5pc}\xymatrix{
  & \ar@/_2pc/[ddl]_{\pi_{EE'}}\R^E \ar[d]^{\pi_{E, E'\cup E''}} \ar@/^2pc/[ddr]^{\pi_{EE''}}& \\
 & \ar[dl]_{\pi_{E'\cup E'', E'}}  \R^{E' \cup E''} \ar[dr]^{\pi_{E'\cup E'', E''}} &  \\
 \R^{E'} & & \R^{E''}}$$

Define $\widetilde{f_1} = f_1 \circ \pi_{E'\cup E'', E'}: \R^{E' \cup E''} \to \R$ and $\widetilde{f_2} = f_2 \circ \pi_{E'\cup E'', E''}: \R^{E' \cup E''} \to \R$, so $\widetilde{F_1} = \pi_{E'\cup E'', E'} ^{\dashv}[Z(f_1)] = \widetilde{f_1}^{\dashv}[\{ 0\}] = Z(\widetilde{f_1}) \subseteq \R^{E'\cup E''}$ and $\widetilde{F_2} = \pi_{E'\cup E'', E''}^{\dashv}[Z(f_2)] = \widetilde{f_2}^{\dashv}[\{ 0\}] = Z(\widetilde{f_2}) \subseteq \R^{E' \cup E''}$ are zerosets.

Note that $\widetilde{F_1}\cap \widetilde{F_2}$ is also a zeroset, namely $\widetilde{F_1}\cap \widetilde{F_2} = Z(\widetilde{f_1}^2 + \widetilde{f_2}^2)$, with $\widetilde{f_1}^2 + \widetilde{f_2}^2 \in \mu_{E' \cup E'',E}^{\dashv}[I]$. In fact, we have the commutative diagram:

$$\xymatrix{
   & \ci(\R^E) &    \\
   & \ci(\R^{E'\cup E''}) \ar[u]^{\mu_{E'\cup E'',E}} & \\
 \ci(\R^{E'}) \ar@/^2pc/[uur]^{\mu_{E'E}} \ar[ur]_{\mu_{E',E'\cup E''}} & & \ci(\R^{E''}) \ar@/_2pc/[uul]_{\mu_{E''E}} \ar[ul]^{\mu_{E'', E'\cup E''}}
}$$

Since the diagram commutes, we have $\mu_{E',E'\cup E''}(f_1) = f_1 \circ \pi_{E'\cup E'', E'} = \widetilde{f_1} \in \mu_{E'\cup E''}^{\dashv}[I]$ and $\mu_{E'',E'\cup E''}(f_2) = f_2 \circ \pi_{E'\cup E'', E''} = \widetilde{f_2} \in \mu_{E'\cup E'',E}^{\dashv}[I]$, so $\widetilde{f_1}^2 + \widetilde{f_2}^2 \in \mu_{E'\cup E'',E}^{\dashv}[I]$.\\


Then $g_1^2 + g_2^2 \in I$ and 

$$\pi_{E, E' \cup E''}^{\dashv}[\widetilde{F_1}\cap \widetilde{F_2}] = \pi_{E, E'\cup E''}^{\dashv}[Z(\widetilde{f_1}^2 + \widetilde{f_2}^2)] = Z(g_1^2 + g_2^2) = G_1 \cap G_2.$$

Let $G \in \widehat{I}$ and $H \in {\cal Z}(\R^E)$ be such that $G \subseteq H$. 
Then there are $g \in I, h \in \ci(\R^E)$ such that $G= Z(g), H = Z(h) \in {\cal Z}(\R^E)$. Now select $E' \subseteq_{fin} E$ and $g' , h' \in \ci(\R^{E'})$ such that $\mu_{E'E}(g') = g, \mu_{E'E}(h') = h$; thus $\pi_{EE'}^\dashv[Z(g')] \subseteq \pi_{EE'}^\dashv[Z(h')]$ and  $g' \in I' := \mu_{E'E}^\dashv[I]$.
Let $G'= Z(g'), H' = Z(h') \in {\cal Z}(\R^{E'})$, then $G' \subseteq H'$.
Since we are dealing with $\ci(\R^{E'})$ with $E'$ finite, Whitney's theorem ({\bf Fact \ref{obs}}) gives us a smooth function, 
$\chi_{{H'}} \in \ci(\R^{E'})$ such that ${H'}=Z(\chi_{{H'}})$. We have $Z(h') = H' = G' \cap H' = Z(g') \cap Z(\chi_{H'}) = Z(g'. \chi_{H'})$ and, since $I' = \mu_{E'E}^\dashv[I]$ is an ideal, $g' . \chi_{H'} \in I'$. 

Since
$H = \pi_{EE'}^\dashv[H']$, $H' = Z(g'. \chi_{H'})$, $g'. \chi_{H'} \in I' $ and
$$\widehat{I} = \{ F \subseteq \R^E \mid (\exists E' \subseteq_{\rm fin} E)(\exists f' \in I'= \mu_{E'}^{\dashv}[I])(F = \pi_{EE'}^{\dashv}[Z(f)])\} \subseteq \wp({\cal Z}(\R^{E})),$$
we have $H \in \widehat{I}$.


\end{proof}










\begin{proposition}Let $E$ be any set. If $\Phi  \subseteq \wp({\cal Z}(\R^E))$ is a filter of zerosets in $\R^E$, then:
$$ \widecheck{\Phi} := \{ f \in \ci(\R^E) \mid Z(f) \in \Phi \} \subseteq \ci(\R^E)$$
\noindent is an ideal of $\ci(\R^E)$.
\end{proposition}

\begin{proof}
Note first that
$$ \widecheck{\Phi} = \{ f \in \ci(\R^E) \mid (\exists E' \subseteq_{\rm fin} E)(\exists {f'} \in \ci(\R^{E'}))((\mu_{E'E}({f'}) = f) \& (\pi_{EE'}^{\dashv}[Z({f'})] \in \Phi))\} \subseteq \ci(\R^E).$$

It is easy to see that $0_E \in \widecheck{\Phi}$. In fact, $Z(0_E) = \R^E \in \Phi$.

Given $f \in \widecheck{\Phi} \subseteq \ci(\R^E)$ and $h \in \ci(\R^E)$.
 Select $E' \subseteq_{fin} E$ and $f', h' \in \ci(\R^{E'})$ such that $\mu_{E'E}(f') = f , \mu_{E'E}(h') = h$. Then $h.f = \mu_{E'E}(h'.f') \in \ci(\R^E)$ and $Z(h.f) = Z(h) \cup Z(f) \supseteq Z(f) \in \Phi$. Thus $h.f \in \widecheck{\Phi}$.

Let $f,g \in \widecheck{\Phi}$. Select $E' \subseteq_{fin} E$ and $f', g' \in \ci(\R^{E'})$ such that $\mu_{E'E}(f') = f , \mu_{E'E}(g') = g$. Thus
$\pi_{EE'}^{\dashv}[Z({f'})], \pi_{EE'}^{\dashv}[Z({g'})] \in \Phi$ and $f+g = \mu_{E'E}(f'+g') \in \ci(\R^E)$.
Since $Z(f + g) \supseteq Z(f)\cap Z(g) \in \Phi$, we obtain $f+g \in \widecheck{\Phi}$.

\end{proof}

\begin{proposition}\label{tatuape}Consider the partially ordered sets:
$$\mathfrak{F} = (\{ \Phi \subseteq \wp({\cal Z}(\R^E)) \mid \Phi\, \text{is a  filter}\}, \subseteq )$$
\noindent and 
$$\mathfrak{I} = (\{ I \subseteq \ci(\R^E) \mid I\,\, \text{is an ideal of}\, \ci(\R^E)\}, \subseteq )$$
The following functions:

$$\begin{array}{cccc}
\vee : & \mathfrak{F} & \rightarrow & \mathfrak{I}\\
   & \Phi & \mapsto & \widecheck{\Phi}
\end{array}$$

$$\begin{array}{cccc}
\wedge : & \mathfrak{I} & \rightarrow & \mathfrak{F}\\
   & I & \mapsto & \widehat{I}
\end{array}$$
form a {\em covariant} Galois connection, $\wedge \dashv \vee$  , that is:
\begin{itemize}
    \item[{\rm (a)}]{Given $\Phi_1, \Phi_2 \in \mathfrak{F}$ such that $\Phi_1 \subseteq \Phi_2$, then $\widecheck{\Phi_1} \subseteq \widecheck{\Phi_2}$;} 
    \item[{\rm (b)}]{Given $I_1, I_2 \in \mathfrak{I}$ such that $I_1 \subseteq I_2$ then $\widehat{I_1} \subseteq \widehat{I_2}$;}
    \item[{\rm (c)}]{For every $\Phi \in \mathfrak{F}$ and every $I \in \mathfrak{I}$ we have:
    $$\widehat{I} \subseteq \Phi \iff I \subseteq \widecheck{\Phi}$$}
    
    Moreover, the mappings $(\vee, \wedge)$ corresponds: \\
    (1) $\wp({\cal Z}(\R^E))$ and $\ci(\R^E)$;\\
    (2) Proper filters of $(\mathfrak{F}, \subseteq)$ and proper ideals of $(\mathfrak{I}, \subseteq)$.
    
\end{itemize}

\end{proposition}
\begin{proof} Items (a), (b), (c) follows directly from the definitions.

Suppose that $\Phi = \wp({\cal Z}(\R^E))$. Then $\widecheck{\Phi} = \{f \in \ci(\R^E)) \mid  f \in \wp({\cal Z}(\R^E))\}$, thus $\widecheck{\Phi} = \ci(\R^E)$. 

Suppose that $I = \ci(\R^E)$. Then 
$\widehat{I} = \{Z(f) \in \wp({\cal Z}(\R^E)) \mid  f \in I\}$, thus $\widehat{I} = \wp({\cal Z}(\R^E))$.

 Suppose that $\Phi$ is a proper filter. If $f \in \ci(\R^E)$ is such that $Z(f) = \varnothing \notin \Phi$, then $f \in \ci(\R^E)^\times$ and $f \notin \widecheck{\Phi} \subseteq \ci(\R^E)$. Thus  $\widecheck{\Phi} \subseteq \ci(\R^E)$ is a proper ideal.

  Suppose that $I$ is a proper ideal. So  $f \notin I$ whenever $f \in \ci(\R^E)^\times$, i.e. whenever  $Z(f) = \emptyset$. Thus $\emptyset \notin \widehat{I}$,  i. e. $\widehat{I}$ is a proper filter.



\end{proof}




\begin{remark} \label{galois-re}
As in any (covariant) Galois connection, we have automatically that:\\

$\bullet$  $I \subseteq  \widecheck{\widehat{I}}$; $\Phi \supseteq  \widehat{\widecheck{\Phi}} $\\

$\bullet$  $\widehat{I} = \widehat{\widecheck{\widehat{I}}}$;  $\widecheck{\Phi}  = \widecheck{\widehat{\widecheck{\Phi}}} $

\end{remark}

The following result gives a more detailed information on these compositions.

\begin{proposition}\label{lapa}Let $I \subseteq \ci(\R^{E})$ be any ideal and $\Phi \subseteq \wp({\cal Z}(\R^E))$ be a filter of zerosets. Then:
\begin{enumerate}

\item   $\widehat{\widecheck{\Phi}} = \{X \subseteq (\R ^E) \mid \exists f \in \ci(\R^E) ( X= Z(f),  Z(f) \in \Phi) \} = \Phi$. 

\item  $\widecheck{\widehat{I}} = \{g \in \ci(\R ^E) \mid \exists f \in \ci(\R^E) ( f \in I, Z(g) = Z(f)) \} = \sqrt[\infty]{I}$.

\end{enumerate}

\end{proposition}
\begin{proof}
Item (1) and the first equality in item (2) follow directly from the definitions.
 We will show that 
 $$\{g \in \ci(\R ^E) \mid \exists f \in \ci(\R^E) ( f \in I, Z(g) = Z(f)) \} = \sqrt[\infty]{I}$$

Note that:
$\sqrt[\infty]{I} = 
\{ g \in \ci(\R^{E}) \mid (\exists {f} \in I )( (\eta_{{g}}({f}) \in \ci(\R^{E})\{ {g}^{-1}\}^{\times})\} =$ \\
$\{ g \in \ci(\R^{E}) \mid (\exists E' \subseteq_{\rm fin} E)(\exists \widetilde{g} \in \ci(\R^{E'}))(\exists \widetilde{f} \in \mu_{E'E}^{\dashv}[I])(g=\widetilde{g}\circ \pi_{EE'})\& (\eta_{\widetilde{g}}(\widetilde{f}) \in \ci(\R^{E'})\{ \widetilde{g}^{-1}\}^{\times})\}$

Given $g \in \widecheck{\widehat{I}}$, there is some finite $E' \subseteq E$, some $\widetilde{g} \in \ci(\R^{E'})$ with $g=\widetilde{g}\circ \pi_{EE'}$ and some $\widetilde{f} \in \mu_{E'E}^{\dashv}[I]$ such that $\pi_{EE'}^{\dashv}[Z(\widetilde{g})] = Z(g) = Z(f) =  \pi_{EE'}^{\dashv}[Z(\widetilde{f})]$. Since $\pi_{EE'}: \R^{E} \to \R^{E'}$ is surjective, we have $Z(\widetilde{g}) = \pi_{EE'}[\pi_{EE'}^{\dashv}[Z(\widetilde{g})]] = \pi_{EE'}[\pi_{EE'}^{\dashv}[Z(\widetilde{f})]] = Z(\widetilde{f})$, so $Z(\widetilde{g}) \supseteq Z(\widetilde{f})$. It follows that $\widetilde{f}\upharpoonright_{\R^{E'}\setminus Z(\widetilde{g})} \in \ci(\R^{E'}\setminus Z(\widetilde{g}))^{\times}$ and, by {\bf Example \ref{frac-exa}}, $\eta_{\widetilde{g}}(\widetilde{f}) \in \ci(\R^{E'})\{ \widetilde{g}^{-1}\}^{\times}$.  Since there is $\widetilde{f} \in \mu_{E'E}^{\dashv}[I]$ such that $\eta_{\widetilde{g}}(\widetilde{f}) \in \ci(\R^{E'})\{ \widetilde{g}^{-1}\}^{\times}$,  it follows that $g \in \sqrt[\infty]{I}$.

Conversely, given $g \in \sqrt[\infty]{I}$, there is some finite $E' \subseteq_{\rm fin} E$, some $\widetilde{g} \in \ci(\R^{E'})$ and some $\widetilde{f} \in \mu_{E'E}^{\dashv}[I]$ such that $g = \widetilde{g} \circ \pi_{E'}$ and $\eta_{\widetilde{g}}(\widetilde{f}) \in \ci(\R^{E'})\{ \widetilde{g}^{-1}\}^{\times}$. So $\widetilde{f}\upharpoonright_{\R^{E'}\setminus Z(\widetilde{g})} \in \ci(\R^{E'}\setminus Z(\widetilde{g}))^{\times}$,  $Z(\widetilde{f})\subseteq Z(\widetilde{g})$ and $\pi_{EE'}^{\dashv}[Z(\widetilde{f})] \subseteq \pi_{EE'}^{\dashv}[Z(\widetilde{g})]$. Since $\pi_{EE'}^{\dashv}[Z(\widetilde{f})] \in \widehat{I}$ and $\widehat{I}$ is a filter, we have $Z(g) = \pi_{EE'}^{\dashv}[Z(\widetilde{g})] \in \widehat{I}$, so $ g  \in \widecheck{\widehat{I}}$.
\end{proof}

\begin{remark}The item (2) in the previous proposition ensures  that the $\ci-$radical of any ideal of a $\ci-$ring is an ideal.
\end{remark}

\begin{proposition}\label{lerigo}Let $E$ be any set, and consider $A=\ci(\R^E)$. The Galois connection $\wedge \dashv \vee$ establishes  bijective correspondences between the:
\begin{itemize}
\item[{\rm (a)}]{The poset of all (proper) filters of zerosets of $\R^E$ and the poset of all (proper) $\ci-$radical ideals of $\ci(\R^E)$, $\mathfrak{I}^{\infty} = \{ I \in \mathfrak{I} \mid \sqrt[\infty]{I}=I\}$;}
\item[{\rm (b)}]{The set of all maximal filters of $(\mathfrak{F}, \subseteq)$ and the set of all maximal ideals of $(\mathfrak{I}, \subseteq)$;}
   \item[{\rm (c)}]{ The poset of all prime (proper) filters of $(\mathfrak{F}, \subseteq)$ and the poset of all prime (proper) $\ci$-radical ideals of $(\mathfrak{I}, \subseteq)$.}

\end{itemize}
\end{proposition}
\begin{proof}

We saw in {\bf Proposition \ref{tatuape}}, that the functions $(\vee, \wedge)$ restricts to maps between proper filters of zerosets of $\R^E$ and proper ideals of $\ci(\R^E)$. Thus the additional parts in items (a) and (c) are automatic.

Ad (a): Let $\Phi$ be a filter of zerosets in $\R^E$, then  by {\bf Proposition \ref{lapa}.(1)} $\Phi = \widehat{\widecheck{\Phi}}$. Let $I$ be a $\ci$-radical ideal in $\ci(\R^E)$, then by {\bf Proposition \ref{lapa}.(2)} and Remark \ref{galois-re}

$$\widecheck{\widehat{\widecheck{\widehat{I}}}} = \widecheck{\widehat{I}}  = {\sqrt[\infty]{I}} =  I$$ 

Thus, since $\widehat{I} = \widehat{\widecheck{\widehat{I}}}$ and   $\widecheck{\Phi}  = \widecheck{\widehat{\widecheck{\Phi}}} $, the (increasing) mappings $(\vee, \wedge)$ establishes a bijective correspondence between
the poset of all filters of zerosets of $\R^E$ and the poset of all $\ci-$radical ideals of $\ci(\R^E)$.

Ad (b): First of all, note that, by a combination of previous results,   if $I$ is a proper ideal of $\ci(\R^E)$, then $\sqrt[\infty]{I}$ is also a proper ideal of $\ci(\R^E)$. Thus if $I$ is a (proper) maximal ideal of $\ci(\R^E)$, then $I = \sqrt[\infty]{I}$.

 Now, by item (a), the increasing mappings $(\vee, \wedge)$ establishes a bijective correspondence between
the poset of all proper filters of zerosets of $\R^E$ and the poset of all proper $\ci-$radical ideals of $\ci(\R^E)$. Thus the mappings $(\vee, \wedge)$  restrict to a pair of inverse bijective correspondence between the set of all maximal filters of zerosets of $\R^E$ and the set of all maximal ideals of $\ci(\R^E)$.

Ad (c): By the bijective correspondence in item (a), it is enough to show that the mappings   $(\vee, \wedge)$  restricts to a pair of mappings  between the set of all prime filters of zerosets of $\R^E$ and the set of all $\ci$-radical prime ideals of $\ci(\R^E)$.

Let $\Phi$ be a prime filter of zerosets of $\R^E$. If $f,g \in \mathcal{C}^{\infty}(\mathbb{R}^E)$ are such that $f \cdot g \in \widecheck{\Phi}$, then $Z(f\cdot g) = Z(f) \cup Z(g) \in \Phi$, so we have $Z(f) \in \Phi$ or $Z(g) \in \Phi$. Thus, $f \in \widecheck{\Phi}$ or $g \in \widecheck{\Phi}$, so $\widecheck{\Phi}$ is a prime ideal of $\ci(\R^E)$; moreover, by item (a), $\widecheck{\Phi}$ is $\ci$-radical.

Let $I$ be a $\ci$-radical prime ideal of $\mathcal{C}^{\infty}(\R^E)$, that is, if $f,g \in \mathcal{C}^{\infty}(\R^E)$ are such that $f \cdot g \in I$ then $f \in I$ or $g \in I$. We need to show that $\widehat{I} = \{ Z(h) | h \in I\}$ is a prime filter of zerosets of $\R^E$.

Let $F=Z(f), G=Z(g), H = Z(h)$ be zerosets of $\R^E$ such that $F \cup G  = H \in \widehat{I}$, $h \in I$. Select $E'\subseteq_{fin} E$ and $f', g', h' \in \ci(\R^{E'})$ such that $f= \mu_{E'E}(f'), g=\mu_{E'E}(g'), h = \mu_{E'E}(h')$. Let $F'=Z(f'), G'=Z(g'), H' = Z(h') \subseteq \R^{E'}$ then $F' \cup G' = H' \in \widehat{I'}$, where $h' \in I' := \mu_{E'E}^\dashv[I]$. 
Since $I$ is  a $\ci$-radical prime ideal of $\mathcal{C}^{\infty}(\R^E)$, then $I'$ is a $\ci$-radical prime ideal of $\mathcal{C}^{\infty}(\R^{E'})$, see {\bf Proposition \ref{egito}.(d)}. {\underline{\em If}} we show that $\widehat{I'}$ is a prime filter of zerosets of $\R^{E'}$ then, we may assume w.l.o.g. that  $Z(f') = F' \in \widehat{I'}$ and  $f' \in \sqrt[\infty]{I'} = I'= \mu_{E'E}^\dashv[I]$;  thus $f = \mu_{E'E}(f') \in I$ and $F = Z(f) \in \widehat{I}$, finishing the proof.

We will prove that $\widehat{I'}$ is a prime filter. We have $Z(f' . g') = Z(f') \cup Z(g') = F' \cup G' = H' = Z(h')$, where $h' \in I'$. Then, $f'.g' \in  \sqrt[\infty]{I'} = I'$. Since $I'$ is a prime ideal, $f' \in I'$ or $g' \in I'$. Thus $F' = Z(f') \in \widehat{I'}$ or $G' = Z(g') \in \widehat{I'}$.

\end{proof}



\section{The order theory of $\ci$-reduced $\ci$-rings}

The results established in the previous section are fundamental to develop an order theory over a broad class $\ci$-rings. In fact, in order to get nice results, we need to assume some technical conditions: the $\ci$-rings must be non-trivial (i.e.  $0 \neq 1$) and  {\bf $\ci$- reduced} (see section 1). 

The fundamental notion here is the following (see \cite{rings1}):

\begin{definition} \label{prec-def} Let  $A$ be a  $\ci-$ring. The {\em canonical relation} on $A$ is

$$\prec_A = \{ (a,b) \in A \times A \mid (\exists u \in A^{\times})(b-a = u^2)\}$$


\end{definition}

\begin{remark} \label{preserve-re} Note that the canonical relation is preserved by $\ci$-homomorphism. In more details: let $A, A'$ be $\ci$-rings and $h : A \to A'$ be a $\ci$-homomorphism. Then for each $a, b \in A$:
$$ a \prec_A b  \ \Rightarrow h(a) \prec_{A'} h(b)$$.
\end{remark}

\begin{proposition} \label{compatibilidade} Let $A$ be any $\ci-$ring. The canonical relation on $A$, $\prec_A$, is compatible with the sum and with the product on $A$.
\end{proposition}
\begin{proof}
Let $a,b \in A$ be such that $a \prec b$ and let $c \in A^\times$ such that $(b - a=c^2)$.

Given any $x \in A$, we have:
$$(b+x)-(a+x) = b - a = c^2,$$
thus $a+x \prec b+x$.

Given $x \in A$ such that $0 \prec x$, one has, since $0 \prec x$, that $(\exists d \in A^{\times})(x=d^2)$. We have, thus:
$$b\cdot x - a \cdot x = (b-a)\cdot x = c^2 \cdot d^2 = (c \cdot d)^2.$$

Since both $c$ amd $d$ are invertible, it follows that $c\cdot d$ is invertible, and $a\cdot x \prec b \cdot x$.
\end{proof}

\begin{proposition} \label{irreflexive}If $A$ is a non trivial $\ci$-ring, then $\prec$, defined as above, is irreflexive, that is, 

$$(\forall a \in A)(\neg(a \prec a))$$
\end{proposition}
\begin{proof}Suppose there is some $a_0 \in A$ such that $a_0 \prec a_0$. By definition, this happens if, and only if there is some $c \in A^{\times}$ such that $0 = a_0 - a_0 = c^2$, so $0$ would be invertible and $0 = 1$.
\end{proof}

To obtain deeper information on the canonical relations, $\prec_A$, it is needed to pass to ``spatial''  specific characterizations of it, by the aid of the results developed in the previous section.  We start this enterprise by the following:

\begin{proposition}\label{mp}Let $A = \frac{\ci(\mathbb{R}^n)}{I}$ be a finitely generated $\ci-$reduced $\ci-$ring. Then:


$$(f+I \prec g + I)\iff ((\exists \varphi \in I)(\forall x \in Z(\varphi))(f(x)<g(x)))$$

\end{proposition}
\begin{proof} This is a direct application of {\bf Corollary \ref{square-cor}}, but we will register here a detailed proof.

Suppose $f+I \prec g + I$, so there is some $h+I \in \left( \frac{\ci(\mathbb{R}^n)}{I}\right)^{\times}$ such that $g-f + I = h^2 + I$. Since $h+I$ is invertible, by \textbf{Proposition \ref{invertivel}} there is some $\psi \in I$ such that:

$$(\forall x \in Z(\psi))(h(x)\neq 0)$$

Since $g-f + I = h^2 + I$, by \textbf{Proposition \ref{equacional}}, there is some $\phi \in I$ such that:

$$(\forall x \in Z(\phi))(g(x)-f(x) = h^2(x)),$$



Taking $\varphi = \phi^2 + \psi^2 \in I$ we have, for every $x \in Z(\psi)\cap Z(\phi) = Z(\varphi)$  both:

$$g(x)-f(x) = h^2(x)$$
and
$$h^2(x) > 0$$

Hence,

$$(\forall x \in Z(\varphi))(f(x)< g(x))$$

Conversely, suppose $f,g \in \ci(\mathbb{R}^n)$ are such that there is some $\varphi \in I$ with satisfying:

$$(\forall x \in Z(\varphi))(f(x)<g(x)).$$

Since $f$ and $g$ are continuous functions, there is an open subset $U \subseteq \R^n$ such that $Z(\varphi) \subseteq U$ and 

$$(\forall x \in U)(f(x)<g(x)).$$

The $\ci$-function:
$$\begin{array}{cccc}
   m: & \R^n & \rightarrow & \mathbb{R}\\
    & x & \mapsto & g(x) - f(x)
\end{array}$$

\noindent is such that $(\forall x \in U)(m(x)>0)$. Thus $m \upharpoonright_{Z(\varphi)}$ is smooth and  $(\forall x \in Z(\varphi))(m(x)>0)$,  so by \textbf{Fact \ref{obs}}  there is a smooth function
$\widetilde{m} \in \ci(\R^n)$ such that  $\widetilde{m}\upharpoonright_{Z(\varphi)} = {m}\upharpoonright_{Z(\varphi)} $ and   $(\forall x \in \R^n)(\widetilde{m}(x) >  0)$.

Now the function $h := \sqrt{\widetilde{m}} : \R^n \to \R$ is a smooth function and since $h \in (\ci(\R^n))^\times$,
we have $h+ I \in \left( \frac{\ci(\mathbb{R}^n)}{I}\right)^{\times}$.

Therefore
$$(\forall x \in Z(\varphi))(g(x)-f(x) = m(x) = \widetilde{m}(x) = h^2(x)))$$

Since $I = \sqrt[\infty]{I}$, by \textbf{Proposition \ref{equacional}} it follows that $(g-f)+I = h^2 + I$ with $h+I$ invertible. Thus, $f+I \prec g+I$.
\end{proof}



\begin{proposition}\label{saf}Given any $\ci-$reduced $\ci-$ring $A$. Then there is a   directed system of its finitely generated $\ci-$rings
$(A_i, \alpha_{ij}: A_i \to A_j)_{\substack{i \in I \\ i \leq j}}$ such that:

\begin{enumerate}
    \item  $A \cong   \varinjlim_{i \in I} A_i;$
    \item For each $i, j$, if $i \leq j$ then $\alpha_{ij} : A_i \to A_j$ and $\alpha_i : A_i \to A$ are injective;
    \item For each $i \in I$, $A_i$ is a $\ci$-reduced  $\ci-$ring;
    \item  For each $a, b \in A$, $a \prec_A b$ iff 
    $\exists i \in I, \exists a_i, b_i \in A_i (\alpha_i(a_i) = a , \alpha_i(b_i) = b \ \text{and} \ a_i \prec_{A_i} b_i).$
\end{enumerate}
\end{proposition}

\begin{proof}

Note that any  $\ci$-reduced  $\ci$-ring can be presented as $A \cong  \ci(\R^E)/I$, where $I = \sqrt[\infty]{I}$ and the latter can be described as a directed colimit of finitely generated $\ci$-reduced $\ci$-rings. Indeed, we have that 
$$\ci(\R^E)/I \cong (\varinjlim_{E' \subseteq_{\rm fin} E} \mathcal{C}^{\infty}(\mathbb{R}^{E'}))/I \cong \varinjlim_{E' \subseteq_{\rm fin} E} (\mathcal{C}^{\infty}(\mathbb{R}^{E'})/\mu_{E',E}^\dashv[I]).$$

 It is clear that $\alpha_{E'} : \ci(\R^{E'})/ \mu_{E',E}^\dashv[I]  \to \ci(\R^E)/I$ is injective, for each $E' \subseteq_{\rm fin} E$. Thus if $E'' \subseteq E'  \subseteq_{\rm fin} E$, then  $\alpha_{E''E'} : \ci(\R^{E''})/ \mu_{E'',E}^\dashv[I]  \to \ci(\R^{E'})/ \mu_{E',E}^\dashv[I]$ is injective too.
 
 We combine the results in {\bf Proposition} \ref{rad-pr}.(a)  and {\bf Proposition} \ref{egito}.(d) to conclude that  $\ci(\R^{E'})/\mu_{E',E}^\dashv[I]$ is a $\ci$-reduced $\ci$-ring. 
 
  Now item (4) follows directly from the definition of canonical relation, since for each $f, g \in \ci(\R^E)$:
  
  $$ \exists u \in \ci(\R^ E) ( (g-f) + I = u^2 +I ; u +I \in (\ci(\R^E)/I)^\times) $$
  
   iff  $ \exists E' \subseteq_{fin} E$, $\exists f',g', u' \in \ci(\R^{E'}), \mu_{E',E}(f') = f, \mu_{E',E}(g') = g,\mu_{E',E}(u') = u$ such that:
  
  $$  (g'-f') + \mu_{E',E}^\dashv[I] = u'^2 + \mu_{E',E}^\dashv[I] \ \text{and} \  u' + \mu_{E',E}^\dashv[I]  \in (\ci(\R^{E'})/\mu_{E',E}^\dashv[I] )^\times ). $$

\end{proof}

We are ready to state and proof the following (very useful) general characterization result of $\prec$:


\begin{theorem}\label{ordem}Let $A = \frac{\ci(\R^E)}{I}$ be a ``general" $\ci-$reduced $\ci-$ring. Let we have:

 $(f+I \prec g + I)\iff ((\exists \varphi \in I)(\forall x \in Z(\varphi))(f(x)<g(x)))$

\begin{proof}

Suppose $f+I \prec g + I$, so there is some $h+I \in \left( \frac{\ci(\mathbb{R}^n)}{I}\right)^{\times}$ such that $g-f + I = h^2 + I$. Since $h+I$ is invertible, by \textbf{Proposition \ref{invertivel2}} there is some $\psi \in I$ such that:

$$(\forall x \in Z(\psi))(h(x)\neq 0)$$

Since $g-f + I = h^2 + I$, by \textbf{Proposition \ref{equacional2}}, there is some $\phi \in I$ such that:

$$(\forall x \in Z(\phi))(g(x)-f(x) = h^2(x)),$$



Taking $\varphi = \phi^2 + \psi^2 \in I$ we have, for every $x \in Z(\psi)\cap Z(\phi) = Z(\varphi)$  both:

$$g(x)-f(x) = h^2(x)$$
and
$$h^2(x) > 0$$

Hence,

$$(\forall x \in Z(\varphi))(f(x)< g(x))$$

Conversely, suppose $f,g \in \ci(\mathbb{R}^E)$ are such that there is some $\varphi \in I$ with satisfying:

$$(\forall x \in Z(\varphi))(f(x)<g(x)).$$

Pick $E' \subseteq_{fin} E$ and $f', g', \varphi' \in \ci(\R^{E'})$ such that $f = \mu_{E'E}(f'), g = \mu_{E'E}(g'), \varphi = \mu_{E'E}(\varphi')$. 

Then $I' := \mu_{E'E}^{\dashv}[I]$ is a $\ci$-radical ideal of $\ci(\R^{E'})$, $\varphi' \in I'$ and  
$$\forall x' \in \R^{E'} ( x' \in Z(\varphi') \to f'(x') < g'(x') ).$$ 

By the finitely generated case  (i.e. {\bf Proposition \ref{mp}}),  
$$f'+I' \prec g'+I'.$$

By (the proof of) {\bf Proposition \ref{saf}.(4)} we obtain

$$f+I \prec g+I.$$

\end{proof}



\end{theorem}




Now, having available a characterization of the canonical relation $\prec$, we can establish many of its  properties.

\begin{proposition}\label{transitive}Let $A$ be any $\ci$-reduced $\mathcal{C}^{\infty}-$ring.  The  canonical relation $\prec_A$ is transitive.

\end{proposition}
\begin{proof}
The $\ci$-reduced  $\ci$-ring $A$ can be presented as $A \cong  \ci(\R^E)/I$, for some set $E$ and some $\ci$-radical ideal  $I = \sqrt[\infty]{I} \subseteq \ci(\R^E)$. Thus let $f, g, h \in \ci(\R^E)$ be such that $f+I \prec g+I \prec h+I$. 

Now apply the characterization {\bf Theorem \ref{ordem}} and consider $\alpha, \beta \in I$ such that
$f(x) < g(x), \forall x \in Z(\alpha)$ and $g(x) < h(x), \forall  x \in Z(\beta)$.

Then $\gamma := \alpha^2 + \beta^2 \in I$ and $Z(\gamma) = Z(\alpha)\cap Z(\beta)$.

Thus $f(x) < g(x) <h(x), \forall x \in Z(\gamma)$ and since $\gamma \in I$, applying again the {\bf Theorem \ref{ordem}} we obtain $f+I \prec h+I$.
\end{proof}

\begin{proposition} \label{asymmetric} Let $A$ be any non-trivial $\ci$-ring and $a,b \in A$. Then  the relation $\prec_A$  is asymmetric, i.e. it holds at most one of the following conditions: $a \prec b$,  $b \prec a$.
\end{proposition}

\begin{proof} Suppose that holds simultaneously both the conditions: $a \prec b$, $b \prec a$. Since $\prec$ is transitive ({\bf Proposition \ref{transitive}}) we have $a \prec a$, but this contradicts {\bf Proposition \ref{irreflexive}},  since $A$ is non-trivial. 

\end{proof}



By a combination of \textbf{Propositions} \ref{irreflexive}, \ref{transitive} (and {\ref{asymmetric}}), the canonical relation $\prec$ on every $\ci$-ring $A$ that is non-trivial and $\ci$-reduced is is irreflexive,  transitive (and asymmetric) bynary relation on $A$ : thus it defines a strict partial order. This motivates the following: 




\begin{definition} \label{parcial-def} Let $A$ be a non-trivial, $\ci$-reduced $\ci-$ring. Then the canonical bynary relation  on $A$, $\prec_A$, ({\bf Definition \ref{prec-def}}) will be  called the ``canonical strict partial order on $A$''.
\end{definition}

Moreover, by {\bf Proposition \ref{compatibilidade}}, it holds:

\begin{theorem} \label{order-teo} Let $A$ be any non-trivial $\ci-$reduced $\ci-$ring. The canonical partial order on $A$, $\prec_A$, is compatible with the sum and with the product on $A$.
\end{theorem}

Note that, due to the above result, to prove the trichotomy of $\prec$ it suffices to prove that holds the ``restricted form of trichotomy'': given any $a \in A$ one has either $a=0$, $a \prec 0$ or $0 \prec a$. But, clearly, this is not true in general:

\begin{example} Let $A = \ci(\R^n)$ and consider the $\ci$-function $f(x_1, \cdots, x_n) := e^{(x_1+ \cdots+x_n)} -1$. 

If $x_1+ \cdots+x_n   \underset{<}{\overset{>}{=}}  0$, then $f(x_1, \cdots, x_n)  \underset{<}{\overset{>}{=}}  0$.

\end{example}

On the other hand, the restricted trichotomy holds for invertible members of some classes of $\ci$-reduced $\ci$-ring:

\begin{proposition}\label{d}Given a $\ci$-reduced $\ci$-ring, $A$, one has:
$$A^{\times} = (A^{\times})^2 \cup (-A^{\times})^2$$
provided $A$ satisfies some of the conditions below:
\begin{enumerate}
    \item $A$ is a free $\ci$-ring; 
    \item $A$ is a $\ci$-reduced $\ci$-domain.
\end{enumerate}

\end{proposition}
\begin{proof}
This hols trivially if $0 = 1$. 
We will prove that for a non trivial $\ci$-reduced $\ci$-ring, $A$,  the non obvious inclusion: $A^{\times} \subseteq (A^{\times})^2 \cup (-A^{\times})^2$ holds.

Item (1): First recall that any free $\ci$-ring is $\ci$-reduced (see {\bf Example \ref{free}}). Let $f \in \ci(\R^E)^\times$, then there is $E' \subseteq_{\rm fin} E$ and $f' \in \ci(\R^{E'})^\times $ such that $f = \mu_{E'E}(f') = f' \circ \pi_{EE'}$. Since $f' : \R^{E'} \to \R$ is continuous and $\R^{E'}$ is connected, then ${\rm range}(f) = {\rm range}(f')$ is a connected subset of $\R$, so it is an interval. Since $0 \notin {\rm range}(f)$, then exactly one of the following alternatives holds: (i) ${\rm range}(f) \subseteq ]-\infty, 0[$ or (ii) ${\rm range}(f) \subseteq ]0, \infty[$. If (i) holds then $f \in -(\ci(\R^E)^\times)^2 $ and if (ii) holds then $f \in (\ci(\R^E)^\times)^2 $.\\

Item (2):  We take a presentation of $A$ as $A \cong \ci(\R^E)/I$, for some set $E$ and some  (ring theoretical) ideal $I \in {\rm Spec}^\infty(A)$.   Let $f \in \ci(\R^E)$ be such that  $(f+I  \in \dfrac{\ci(\R^E)}{I}^{\times})$. By {\bf Proposition \ref{invertivel2}}, there exists $\varphi \in I$ such that: 
$$(\forall x \in \R^{E})(x \in Z(\varphi) \to f(x) \neq 0 ).$$

Let $E' \subseteq_{\rm fin} E$ be such that $\varphi = \varphi' \circ \pi_{EE'}$ and $f = f' \circ \pi_{EE'}$, for some $\varphi', f' \in \ci(\R^{E'})$.  Then

$$(\forall x' \in \R^{E'})(x' \in Z(\varphi') \to f'(x') \neq 0 ).$$

Thus,

$$Z(\varphi') \subseteq [f'>0] \cup [-f'>0],$$

\noindent where: $[\pm f' > 0 ] := \{ x' \in \R^{E'} : \pm f'(x') > 0\}$.

Note that:

$$ Z(\varphi') \cap [\pm f' \geq 0] = Z(\varphi') \cap [\pm f'>0] $$

Since $f'$ is a continuous function, $[\pm f' \geq 0] = (\pm f')^{\dashv}[[0, \infty[ ]$ is a closed subset of $\R^{E'}$, and by {\bf Fact \ref{char}},  there is some  $\chi_{\pm} \in \ci(\R^{E'})$ such that 

$$Z (\chi_{\pm}) = [ \pm f' \geq 0].$$

Thus,

$$Z(\varphi') = Z(\varphi') \cap \R^{E'} =  Z(\varphi') \cap (Z(\chi_-) \cup Z(\chi_{+})) =$$

$$ = (Z(\varphi') \cap Z(\chi_{-})) \cup (Z(\varphi') \cap Z(\chi_+)) = Z({\varphi'}^2 + \chi_{-}^2) \cup Z({\varphi'}^2 + \chi_+^2).$$

Since $I$ is a $\ci$-radical (proper) prime ideal of $\ci(\R^E)$, then $I':= \mu_{E'E}^\dashv[I] \subseteq \ci(\R^{E'})$ is a  $\ci$-radical (proper) prime ideal of $\ci(\R^{E'})$ (see {\bf Proposition \ref{egito}.(d)}). 

By {\bf Proposition \ref{lerigo}}, $I'$ corresponds to a prime filter (of zero sets) $\widehat{I'}$ and since
$Z({\varphi'}^2 + \chi_{-}^2) \cup Z({\varphi'}^2 + \chi_+^2) = Z(\varphi')$
 and $\varphi' \in I'$, then some of the subsets $ Z({\varphi'}^2 + \chi_{-}^2),  Z({\varphi'}^2 + \chi_+^2)$ belongs to the $\ci$-radical $I'$.
 By {\bf Proposition \ref{lapa}}:
$\widecheck{\widehat{I'}} = \sqrt[\infty]{I'} = I'$, thus some of the functions $({\varphi'}^2 + \chi_{-}^2),  ({\varphi'}^2 + \chi_+^2)$ belongs to  $I'$.

Now recall that:
 
 $$ Z({\varphi'}^2 + \chi_{\pm}^2) = Z(\varphi') \cap Z(\chi_{\pm})  = Z(\varphi') \cap [\pm f' \geq 0] =   Z(\varphi') \cap [\pm f' > 0] $$
 
 and consider $\alpha_\pm := ({\varphi'}^2 + \chi_{\pm}^2) \circ \pi_{EE'} \in \ci(\R^E).$\\
 
 Then some of the alternatives holds:\\
 
 (i) $\alpha_- \in I$ and $(\forall x \in  Z(\alpha_-)) (-f(x) >0)$; \\
 
 (ii)  $\alpha_+ \in I$ and $(\forall x \in  Z(\alpha_+)) (f(x) >0)$. \\




Applying {\bf Theorem \ref{ordem}}, if (i) holds then $f + I \in  -((\ci(\R^E)/I)^\times)^2 $ and if (ii) holds then $f +I  \in ((\ci(\R^E)/I)^\times)^2$.

This establishes the desired inclusion:

$$A^{\times}  \subseteq  (A^{\times})^2
\overset{.}{\cup} 
(-A^{\times})^2.$$




\end{proof}

There is another natural way to consider that a partial order $\preceq$ (where $a \preceq b$ iff $a \prec b$ or $a = b$)  is compatible with sums: if $0 \preceq x$ and $0 \preceq y$, then $0 \preceq x+y$. This one also holds, as it follows (directly) from the results obtained below.

\begin{proposition}\label{b}Given any $\ci$-reduced $\ci$-ring $A$, denote by $(A^\times)^2 = A^2 \cap A^{\times} = (A^2)^\times$. Then the following hold:

\begin{enumerate}
    \item $\left( \sum A^2 \right) \cap A^{\times} = (A^\times)^2$
    
    \item  $(A^\times)^2 + \sum A^2 = (A^\times)^2$
    
    \item  $\sum (A^{\times})^2 = (A^{\times})^2$
\end{enumerate}

\end{proposition}
\begin{proof}
 {\bf First equality:}

One easily checks that:
$$\left( \sum A^2 \right) \cap A^{\times} \supseteq A^2 \cap A^{\times},$$

\noindent so we only need to prove the opposite inclusion.

We know that $A \cong \dfrac{\ci(\R^E)}{I}$ for some  set $E$  and some $\ci$-radical ideal $\sqrt[\infty]{I} = I \subseteq \ci(\R^E)$. Let $f \in \ci(\R^E)$ be such that $q_I(f)= f+I \in \left( \sum \dfrac{\ci(\R^E)}{I}^2 \right) \cap \dfrac{\ci(\R^E)}{I}^{\times}$.  Since $q_I(f) \in \left(\dfrac{\ci(\R^E)}{I}\right)^{\times}$, by  {\bf Proposition \ref{invertivel2}}, there is some $\varphi \in I$ such that:
$$(\forall x \in Z(\varphi))(f(x)\neq 0)$$
\noindent and since $q_I(f) \in \sum \dfrac{\ci(\R^E)}{I}^2$, by  {\bf Proposition \ref{equacional2}}, there are $f_1, \cdots, f_k \in \ci(\R^E)$ and $\psi \in I$ such that:

$$(\forall x \in Z(\psi))(f(x)=f_1(x)^2+ \cdots + f_k^2(x) \geq 0),$$

Thus $\varphi^2 + \psi^2 \in I$ and

$$(\forall x \in Z(\varphi^2+\psi^2) = Z(\varphi)\cap Z(\psi)) (f(x) = f_1^2(x)+\cdots + f_k^2(x)>0).$$

Applying {\bf Theorem \ref{ordem}}, we have: 
$$ 0 + I \prec f+I,$$
so $f +I = u^2 +I$, for some $u \in \left(\dfrac{\ci(\R^E)}{I}\right)^{\times}$, establishing the equality in item (1).\\

{\bf Second and third equalities}: 

One easily checks that:
$$  (A^\times)^2 + \sum A^2 \supseteq \sum (A^{\times})^2 \supseteq (A^\times)^2,$$

\noindent so, to establish the items  (2) and (3),  we only need to prove that 

$$  (A^\times)^2 + \sum A^2 \subseteq  (A^\times)^2.$$

Present $A$ as $A \cong \dfrac{\ci(\R^E)}{I}$ for some  set $E$  and some $\ci$-radical ideal $\sqrt[\infty]{I} = I \subseteq \ci(\R^E)$. Let $f \in \ci(\R^E)$ such that $q_I(f)= f+I \in    {\left(\dfrac{\ci(\R^E)}{I}^\times \right)}^2 +  \left( \sum \dfrac{\ci(\R^E)}{I}^2 \right)$. I.e., there are $g, h_1, \cdots, h_k \in \ci(\R^E)$ such that $f+I = (g^2 + h_1^2 + \cdots h_k^2) + I$ and $g+I \in {\left(\dfrac{\ci(\R^E)}{I}^\times \right)}$.

Applying {\bf Proposition \ref{equacional2}} and {\bf Theorem \ref{invertivel2}}, we conclude that there is $\theta \in I$ such that:

    $$(\forall x \in Z(\theta))(f(x)= g(x)^2 + h_1(x)^2+ \cdots + h_k^2(x)\ \text{and} \ g(x) \neq 0 )$$
    
    Thus $$ (\forall x \in Z(\theta))(f(x) >0).$$

Since $\theta \in I$, applying {\bf Theorem \ref{ordem}}, we have: 
$$ 0 + I \prec f+I,$$
so $f +I = u^2 +I$, for some $u \in \left(\dfrac{\ci(\R^E)}{I}\right)^{\times}$, establishing the desired inclusion.

\end{proof}

From the second equality above, it follows directly the:

\begin{corollary} \label{boundedinv} Every $\ci$-ring  $A$ has the ``weak bounded inversion property'' (definition 7.1. in \cite{DM}), i.e. $1 + \sum A^2 \subseteq A^\times$. 

\end{corollary}

\section{On the order theory of {$\ci$-fields} and applications}
         	

In this section we use the results on the previous section to continue the development of  the order theory of $\ci$-fields initiated in \cite{rings1} and to apply these results to introduce another aspect of the order theory of general $\ci$-rings.

We start  presenting some  results  concerning   $\ci$-fields that will be useful in the sequel.


\begin{proposition} \label{fieldfg-prop} Let $A$ be a $\mathcal{C}^{\infty}-$field. If $A$ is a finitely generated $\ci$-ring, then $\R \cong A$.
 \end{proposition}

\begin{proof} Let $A \cong  \dfrac{\mathcal{C}^{\infty}(\mathbb{R}^n)}{I}$ be a finitely generated $\mathcal{C}^{\infty}-$field, so $I$ must be a maximal ideal of $\mathcal{C}^{\infty}(\R^n)$ (otherwise the quotient would not even be an ordinary field). By \textbf{Proposition 3.6} of \cite{MRmodels}, every maximal ideal $I$ of $\mathcal{C}^{\infty}(\R^n)$  has the form $I = \ker {\rm ev}_x $, where $x \in Z(I)$ and
$$\begin{array}{cccc}
{\rm ev}_x: & \mathcal{C}^{\infty}(\R^n) & \rightarrow & \mathbb{R} \\
            & f                                  & \mapsto     & f(x)
\end{array}$$
Since ${\rm ev}_x$ is surjective, by the \textbf{Fundamental Theorem of the $\mathcal{C}^{\infty}-$Homomorphism} we have:
$$A \cong \dfrac{\mathcal{C}^{\infty}(\R^n)}{I} = \dfrac{\mathcal{C}^{\infty}(\R^n)}{\ker {\rm ev}_x} \cong \R$$
\end{proof}

\begin{proposition} \label{field-prop} (see \cite{BM2}) Let $A$ be a 
          $\mathcal{C}^{\infty}-$ring.
          \begin{enumerate}
              \item If  $A$  is a 
          $\mathcal{C}^{\infty}-$field, then $A$ is a $\mathcal{C}^{\infty}-$reduced $\mathcal{C}^{\infty}-$domain.
          \item If $A$ is a $\mathcal{C}^{\infty}-$reduced $\mathcal{C}^{\infty}-$domain, then $A\{ (A \setminus\{0\})^{-1}\}$ is a $\ci$-field, $\eta_{A \setminus\{0\}} : A \rightarrowtail A\{ (A \setminus\{0\})^{-1}\}$ is an injective $\ci$-homomorphism  and $\eta_{A \setminus\{0\}}$ is universal among the injective $\ci$-homomorphisms from $A$ into some $\ci$-field.
         	\item $A$ is isomorphic to a  $\ci$-subring of a $\ci$-field if, and only if $A$ is a $\mathcal{C}^{\infty}-$reduced $\mathcal{C}^{\infty}-$domain.
         	\item For every  proper prime ideal $\mathfrak{p} \subseteq A$ that is $\mathcal{C}^{\infty}-$radical, we have a canonical  $\ci$-field $k_{\mathfrak{p}}(A) := $ $\dfrac{A}{\mathfrak{p}} \{ q_{\mathfrak{p}} [A \setminus \mathfrak{p}]^{-1} \}$ and a canonical $\ci$-morphism with kernel $\mathfrak{p}$, $A \  \overset{q_{\mathfrak{p}}}\twoheadrightarrow \ \dfrac{A}{\mathfrak{p}} \ \overset{\eta_{q_{\mathfrak{p}}[{A \setminus \mathfrak{p}]}}}\rightarrowtail \  k_{\mathfrak{p}}(A)$.
          \end{enumerate}
\end{proposition}

Now we are ready to turn our attention to the order theory of $\ci$-fields.

By \textbf{Corollary \ref{boundedinv}}, every $\ci$-reduced $\ci$-ring  $A$ satisfies the relation $1 + \sum A^2 \subseteq A^\times$. 
In particular, every $\ci$-field $A$ is formally real, i.e. $-1 \notin \sum A^2$, thus it can be endowed with some linear order relation compatible with its sum and product. In fact, since a $\ci$-field is a non-trivial $\ci$-reduced $\ci$-ring, we have a distinguished linear order relation in $A$ that is compatible with its sum and product:

\begin{theorem} \label{orderedfield-teo} Let $A$ be a $\ci-$field and $\prec_A$ be the canonical strict partial order on $A$, cf. {\bf Definition \ref{parcial-def}}.  Then $(A, \preceq)$ is a totally/linearly ordered field, {\em i.e.}, $\preceq$ is a reflexive, transitive and anti-symmetric binary relation in $A$ that is compatible with sum and product and, moreover,  it holds the {\em trichotomy law}, {\em i.e.}, for every $a,b \in A$ we have exactly one of the following $a = b$ or $a \prec b$ or $b \prec a$.
\end{theorem}

\begin{proof}
Since $a \preceq b$ iff $a \prec b $ or $a = b$, it follows directly from {\bf Theorem \ref{order-teo}} that $\preceq$ is a reflexive, transitive and anti-symmetric binary relation in $A$ that is compatible with sum and product.

By the compatibility of $\prec$ with the sum, to obtain the trichotomy law it is enough to show that for every $f \in A \setminus \{ 0\}$ we have either $(0 \prec f)$  or $(f \prec 0).$

Since $A$ is a $\ci$-field, $A^\times = A \setminus \{ 0\}$ and the result follows directly from {\bf Proposition \ref{field-prop}.(1)} and {\bf Proposition \ref{d}.(2)}: $ A^{\times} = (A^{\times})^2 \cup (-A^{\times})^2$.




\end{proof}





 In general, a field could support many linear orders compatible with its sum and product. A field is  called {\bf Euclidean} if it has a \underline{unique} (linear, compatible with $+, \cdot$) ordering: these fields are precisely the ordered fields such that every positive member has an square root in the field.  It is clear from the definition of $\prec$ and by {\bf Theorem \ref{orderedfield-teo}} that every $\ci$-field is Euclidean.\\
         	
Now recall that a totally ordered field $(F, \leq)$ is \textbf{real closed} if it satisfies the following two conditions:
\begin{itemize}
  \item[(a)]{$(\forall x \in F)(0 \leq x  \rightarrow (\exists y \in F)(x=y^2))$ (i.e. it is an Euclidean field);}
  \item[(b)]{every polynomial of odd degree has, at least, one root;}
\end{itemize}

Equivalently, a totally ordered field $(F, \leq)$ is real closed if, and only if it  satisfies the conclusion of {\bf intermediate value theorem} for all {\em polynomial} functions $h : F \to F$.\\

As pointed out in {\bf Theorem 2.10} of \cite{rings1}, it holds the following:

\begin{fact}Every $\mathcal{C}^{\infty}-$field, $A$, together with its canonical order $\prec$  is such that $A$ is a real closed field.
\end{fact}


In fact, a stronger property holds for every $\ci$-fields We have the $\mathcal{C}^{\infty}-$analog of the notion of ``real closedness'':

\begin{fact}[\textbf{Theorem 2.10'} of \cite{rings1}]Let $(F,\prec)$ be a $\mathcal{C}^{\infty}-$field. Then $(F,\prec)$ is \index{$\mathcal{C}^{\infty}-$real closed}\textbf{$\mathcal{C}^{\infty}-$real closed}. I.e., it holds:\\
$$(\forall f \in F\{ x\})((f(0)\cdot f(1)<0)\&(1 \in \langle \{ f, f'\}\rangle \subseteq F\{ x\} ) \rightarrow$$ 
$$(\exists \alpha \in ]0,1[ \subseteq F)(f(\alpha)=0))$$
\end{fact}


Note that the class of $\ci$-fields is an $\mathcal{L}$-elementary (proper) class, where  all structures have cardinally at least $2^{\aleph_0}$, but every finitely generated $\ci$-field is real closed  archimedian, i.e. isomorphic to $\R$ (see {\bf Proposition \ref{fieldfg-prop}}).\\

The notion of $\ci$-field is also useful to analyze the order theory of $\ci$-reduced $\ci$-rings:

\begin{remark} \label{local-global fg} Consider $A = \ci(\R^n)$. The inclusion $i  : \ci(\R^n) \hookrightarrow  {\rm Func}(\R^n, \R) = {\R}^{\R^n}$ obviously preserves and reflects the equality relation (=) and the canonical strict partial order ($\prec$). Note that, by  {\bf Proposition \ref{fieldfg-prop}}, this inclusion  can be identified with 

$$A \  \overset{(q_{\mathfrak{m}})_{\mathfrak{m}}}\rightarrow \ \prod_{\mathfrak{m} \in {\rm Max}(A)} \dfrac{A}{\mathfrak{m}}$$

Thus, the family of all $\ci$-fields $\{ \dfrac{A}{\mathfrak{m}} : \mathfrak{m}\in {\rm Max}(A)\}$  encodes the  canonical relation  $\prec_A$ on $A$. 

The family of all $\ci$-fields $\{k_{\mathfrak{p}}(A) : \mathfrak{p}\in {\rm Spec}^{\infty}(A) \}$ also encodes the  canonical relation  $\prec_A$ on $A$. 

Consider the canonical $\ci$-homomorphism $c_A : A \to \prod_{\mathfrak{p}\in {\rm Spec}^{\infty}\,(A)} k_{\mathfrak{p}}(A)$, that is given by:

$$A \  \overset{(q_{\mathfrak{p}})_{\mathfrak{p}}}\rightarrow \ \prod_{\mathfrak{p}\in {\rm Spec}^{\infty}\,(A)} \dfrac{A}{\mathfrak{p}} \ \overset{(\eta_{q_{\mathfrak{p}}[{A \setminus \mathfrak{p}]}})_{\mathfrak{p}}}\rightarrowtail \  \prod_{\mathfrak{p}\in {\rm Spec}^{\infty}\,(A)} k_{\mathfrak{p}}(A)$$

By {\bf Proposition \ref{field-prop}.(4)}, $\ker(c_A) = \bigcap_{\mathfrak{p} \in {\rm Spec}^{\infty}\,(A)} \mathfrak{p}$. Thus, by {\bf Theorem \ref{TS}.(e)},  $\ker(c_a) = \sqrt[\infty]{(0)} = \{0\}$  and $c_A$ is an injective $\ci$-homomorphism, i.e., it preserves and {\em reflects} the equality relation. We will see that $c_A$ also preserves and reflects the canonical relation $\prec$.

The $\ci$-homomorphism $c_A$ preserves $\prec$ (see {\bf Remark \ref{preserve-re}}). Note that, to establish that $c_A : A \to \prod_{\mathfrak{p}\in {\rm Spec}^{\infty}\,(A)} k_{\mathfrak{p}}(A)$ reflects $\prec$ it suffices to certify that: 

$$A \  \overset{(q_{\mathfrak{m}})_{\mathfrak{m}}}\rightarrow \ \prod_{\mathfrak{m}\in {\rm Max}(A)} \dfrac{A}{\mathfrak{m}}$$

\noindent reflects $\prec$. In fact, since $(\forall \mathfrak{m}\in {\rm Max}(A)) (\dfrac{A}{\mathfrak{m}} \cong k_{\mathfrak{m}}(A))$, the inclusion ${\rm Max}(A) \subseteq {\rm Spec}^{\infty}\,(A)$ (this holds by {\bf Proposition \ref{field-prop}.(1)}) induces a canonical ``projection":

$$ \pi_A :  \prod_{\mathfrak{p}\in {\rm Spec}^{\infty}\,(A)} k_{\mathfrak{p}}(A) \twoheadrightarrow \prod_{\mathfrak{m}\in {\rm Max}(A)} \dfrac{A}{\mathfrak{m}} $$

\noindent and, obviously, 

$$ (q_{\mathfrak{m}})_{\mathfrak{m}} = \pi_A \circ c_A ,$$

Thus,  if $a , b \in A$ are such that $a \nprec b$, implies  $ (q_{\mathfrak{m}})_{\mathfrak{m}}(a) \nprec (q_{\mathfrak{m}})_{\mathfrak{m}}(b)$, then also holds $c_A(a) \nprec c_A(b)$.

\end{remark}

\vspace{0.3cm}

Now we will apply the results on $\ci$-fields to describe another approach of the order theory of (general) $\ci$-rings.

\begin{definition} \label{rel}
Let $A$  be an arbitrary $\ci$-ring.
Let $\mathcal{F}$ be the (proper) class of all the $\mathcal{C}^{\infty}-$homomorphisms of  $A$ to some $\mathcal{C}^{\infty}-$field. We define the following relation $\mathcal{R}$: given $h_1: A \to F_1$ and $h_2: A \to F_2$, we say that $h_1$ is related with $h_2$ if, and only if, there is some $\mathcal{C}^{\infty}-$field $\widetilde{F}$ and some $\mathcal{C}^{\infty}-$fields homomorphisms $\mathcal{C}^{\infty}$ $f_1: F_1 \to \widetilde{F}$ and $f_2: F_2 \to \widetilde{F}$ such that the following diagram commutes:

$$\xymatrixcolsep{5pc}\xymatrix{
  & F_1 \ar[dr]^{f_1} & \\
A \ar[ur]^{h_1} \ar[dr]_{h_2} & & \widetilde{F} \\
  & F_2 \ar[ur]_{f_2}
}$$

The relation $\mathcal{R}$ defined above is symmetric and reflexive.
\end{definition}


The above considerations prove the following:

\begin{proposition}[see \textbf{Proposition 54 of \cite{BM2}}]If $h_1: A \to F_1$ and $h_2: A \to F_2$ be two $\mathcal{C}^{\infty}-$homomorphisms from the $\mathcal{C}^{\infty}-$ring  $A$ to the $\mathcal{C}^{\infty}-$fields $F_1,F_2$ such that $(h_1,h_2) \in \mathcal{R}$, then $\ker(h_1)=\ker(h_2)$.
\end{proposition}


\begin{definition}Let $A$ be a $\mathcal{C}^{\infty}-$ring. A \index{$\mathcal{C}^{\infty}-$ordering}$\mathcal{C}^{\infty}-$\textbf{ordering} in $A$ is a subset $P \subseteq A$ such that:
\begin{itemize}
    \item [(O1)]{$P + P \subseteq P$;} 
    \item [(O2)]{$P \cdot P \subseteq P$;} 
    \item [(O3)]{$P \cup (-P) = A$}
  \item [(O4)]{$P \cap (-P) = \mathfrak{p} \in {\rm Spec}^{\infty}\,(A)$}
\end{itemize}
\end{definition}


\begin{fact}  Let $\Sigma(A) := \{ (\mathfrak{p}, Q): \mathfrak{p} \in {\rm Sper}^{\infty}\,(A),  Q \in {\rm Spec}^{\infty}\,(k_{\mathfrak{p}}(A))\}$. The  mapping $P \in {\rm Spec}^{\infty}\,(A) \mapsto (\mathfrak{p}_P, Q_P) \in \Sigma(A)$, where $\mathfrak{p}_P := P \cap (-P)$ (or simply $\mathfrak{p}$) and $Q_P := \{ \eta_{q_{\mathfrak{p}}[{A \setminus \mathfrak{p}]}}(a + \mathfrak{p}) . (\eta_{q_{\mathfrak{p}}[{A \setminus \mathfrak{p}]}}(b+ \mathfrak{p}))^{-1} : b \notin \mathfrak{p}, a.b \in P\} $  is a bijection (this is essentially the {\bf Proposition 5.1.1} in \cite{Mar}).

\end{fact}

\begin{definition}Let $A$ be a $\mathcal{C}^{\infty}-$ring. Given a $\mathcal{C}^{\infty}-$ordering $P$ in $A$, the \index{$\mathcal{C}^{\infty}-$support}$\mathcal{C}^{\infty}-$\textbf{support} of $A$ is given by:

$${\rm supp}^{\infty}(P):= \mathfrak{p}_P = P \cap (-P)$$
\end{definition}

\begin{definition}\label{Santo}Let $A$ be a $\mathcal{C}^{\infty}-$ring. The \index{$\mathcal{C}^{\infty}-$real spectrum of  $A$}\textbf{$\mathcal{C}^{\infty}-$real spectrum of  $A$} is given by:

$${\rm Sper}^{\infty}\,(A)=\{ P \subseteq A | P \, \mbox{is an ordering of the elements of}\,\, A\}$$
together with the (spectral) topology generated by the sets:

$$H^{\infty}(a) = \{ P \in {\rm Sper}^{\infty}\,(A) | a \in P \setminus {\rm supp}^{\infty}\,(P)\}$$

\noindent for every $a \in A$. The topology generated by these sets will be called ``smooth Harrison topology'', and will be denoted by ${\rm Har}^{\infty}$.
\end{definition}

\begin{remark}
The suitable notion of prime spectrum of a $\ci$-ring $A$, ${\rm Spec}^\infty(A)$, appeared for the  first time in  \cite{rings2}: this is the main spatial notion to develop ``Smooth Algebraic Geometry". On the other hand,  in \cite{tese} was introduced the notion of smooth real spectrum of a   $\ci$-ring $A$, ${\rm Sper}^\infty(A)$: this seems to be  the suitable spatial notion for the development of ``Smooth \underline{ Real} Algebraic Geometry".

\end{remark}

\begin{fact}\label{avio}Given a $\mathcal{C}^{\infty}-$ring $A$, we have a  function given by:

$$\begin{array}{cccc}
    {\rm supp}^{\infty}: & ({\rm Sper}^{\infty}(A), {\rm Har}^{\infty}) & \rightarrow & ({\rm Spec}^{\infty}\,(A), {\rm Zar}^{\infty}) \\
     & P & \mapsto & P \cap (-P)
  \end{array}$$

\noindent which is spectral, and thus continuous, since given any $a \in A$, ${{\rm supp}^{\infty}}^{\dashv}[D^{\infty}(a)] = H^{\infty}(a)\cup H^{\infty}(-a)$.
\end{fact}

\underline{Unlike what happens in ordinary} \underline{Commutative} \underline{Algebra}, we have the following (stronger) result in "Smooth Commutative Algebra", as a consequence of the fact  that every $\ci$-field is ($\ci$-)real closed\footnote{In fact, to obtain this result it is enough to know that every $\ci$-field is {\em Euclidean}.} and some separation theorems (see \cite{separation} or \textbf{Theorem 45} of \cite{BM2}):

\begin{theorem} For each  $\mathcal{C}^{\infty}-$ring $A$, the mapping

$$\begin{array}{cccc}
    {\rm supp}^{\infty}: & ({\rm Sper}^{\infty}(A), {\rm Har}^{\infty}) & \rightarrow & ({\rm Spec}^{\infty}\,(A), {\rm Zar}^{\infty})
  \end{array}$$

is a (spectral) {\underline{bijection}}.

\end{theorem}

\section{Concluding remarks and future works}

\begin{remark}

 It is natural to ask if the class of $\mathcal{C}^{\infty}$-fields is model-complete in the language of $\mathcal{C}^{\infty}$-rings or even admits elimination of quantifiers (possibly in the language expanded by a unary predicate for the positive cone of an ordering). If the former holds, then the relation  ${\cal R}$ between pairs of morphism with the same source and target $\mathcal{C}^{\infty}$-fields, that encodes ${\rm Sper}^\infty$, is already a transitive relation (as it occurs in the algebraic case).
 
 \end{remark}

\begin{remark}
 If the class of $\ci$-fields admits quantifier elimination (over a reasonable language), then  it is possible to adapt the definition and results provided in \cite{Robson} on ``Model-theoretic Spectra'' and describe ``logically'' the spectral topological spaces ${\rm Spec}^\infty(A)$ and/or ${\rm Sper}^\infty(A)$ as certain equivalence classes of homomorphisms from $A$ into models of a  ``nice'' theory $T$. Moreover, since the techniques in this work provide  structural sheaves of ``definable functions'', we could compare them with other ones previously defined and determine other new natural model-theoretic spectra in $\mathcal{C}^{\infty}$-structures.

\end{remark}

\begin{remark}

Another evidence that a systematic  model-theoretic analysis of $\ci$-rings, (not only under  real algebra perspective but also under differential algebra perspective), should be interesting and deserves a further attention is indicated in \cite{FM}. In that work, were given the first steps towards a  model-theoretic connection between three kinds of structures: o-minimal structures, Hardy fields and smooth rings. This triple is related to  another one -- Hardy fields, surreal numbers and transseries -- studied in \cite{vandendries-matthias-}: 
these are linked by the notion of H-field  which provides a common framework for these structures. They present a  model-theoretic analysis of the category of
H-fields, e.g. the theory of H-closed fields is model complete,  and relate these results with the latter tripod, that according the authors  M. Aschenbrenner, L. v. Dries and J. v. Hoeven (\cite{vandendries-matthias-}):
 {{\em ``...are three ways
to enrich the real continuum by infinitesimal and infinite quantities. Each of these
comes with naturally interacting notions of ordering and derivative''.}}



\end{remark}

\end{document}